\documentclass[11pt, a4paper]{amsart}

\usepackage[T1]{fontenc}
\usepackage{lmodern}
\usepackage{microtype}
\usepackage{enumerate}	

\usepackage[final]{hyperref}

\usepackage[abbrev]{amsrefs}

\usepackage{amssymb}

\newtheorem{theorem}{Theorem}
\newtheorem{proposition}{Proposition}[section]
\newtheorem{claim}{Claim}
\newtheorem{lemma}[proposition]{Lemma}
\newtheorem{corollary}[proposition]{Corollary}
\theoremstyle{definition}

\numberwithin{equation}{section}
\newenvironment{proofclaim}
                [1]
                [Proof of the claim]
                {\begin{proof}[#1]}
                {\end{proof}}

\newcommand{\N}{{\mathbb N}}
\newcommand{\Rset}{{\mathbb R}}
\newcommand{\Bset}{{\mathbb B}}

\newcommand{\st}{\;:\;}

\newcommand{\abs}[1]{\lvert #1 \rvert}
\newcommand{\bigabs}[1]{\bigl\lvert #1 \bigr\rvert}
\newcommand{\Bigabs}[1]{\Bigl\lvert #1 \Bigr\rvert}
\newcommand{\dualprod}[2]{\langle #1, #2 \rangle}
\newcommand{\bigscalprod}[2]{#1\cdot #2}%{\bigl( #1 \big\vert #2 \bigr)}
\newcommand{\scalprod}[2]{#1\cdot #2}%{( #1 \vert #2 )}

\newcommand{\dif}{\,\mathrm{d}}

\DeclareMathOperator{\supp}{supp}
\DeclareMathOperator{\dist}{dist}

\allowdisplaybreaks

\title{Nodal solutions for the Choquard equation}

\author{Marco Ghimenti}
\address{Universit\`a di Pisa\\
Dipartimento di Matematica\\
Largo B. Pontecorvo 5\\
56100 Pisa\\
Italy}

\email{marco.ghimenti@dma.unipi.it}

\author{Jean Van Schaftingen}
\address{Universit\'e Catholique de Louvain\\ 
Institut de Recherche en Math\'ematique et Phy\-sique\\ 
Chemin du Cyclotron 2 
bte L7.01.01\\ 1348 Louvain-la-Neuve \\ 
Belgium}
\email{Jean.VanSchaftingen@uclouvain.be}

\keywords{Stationary nonlinear Schr\"odinger--Newton equation; 
stationary Hartree equation;
nodal Nehari set;
concentration-compactness.}

\subjclass[2010]{35J91 (35J20)}

\date{January 29, 2016}

\begin{document}

\begin{abstract}
We consider the general Choquard equations
\begin{equation*}
 -\Delta u + u  = \bigl(I_\alpha \ast \abs{u}^p\bigr) \abs{u}^{p - 2} u 
\end{equation*}
where \(I_\alpha\) is a Riesz potential.
We construct minimal action odd solutions for \(p \in (\frac{N + \alpha}{N}, 
\frac{N + \alpha}{N - 2})\) and minimal action nodal solutions for \(p 
\in (2,\frac{N + \alpha}{N - 2})\).
We introduce a new minimax principle for least action nodal solutions and 
we develop new concentration-compactness lemmas for sign-changing Palais--Smale sequences.
The nonlinear Schr\"odinger equation, 
which is the nonlocal counterpart of the Choquard equation,
does not have such solutions.
\end{abstract}

\maketitle

\tableofcontents

\section{Introduction}
We study the general Choquard equation
\begin{equation}
\label{eqChoquard}
\tag{$\mathcal{C}$}
 -\Delta u + u  = \bigl(I_\alpha \ast \abs{u}^p\bigr) \abs{u}^{p - 2} u \quad\text{in \(\Rset^N\)},
\end{equation}
where \(N \ge 1\), \(\alpha \in (0, N)\) and \(I_\alpha : \Rset^N \to \Rset\) is the Riesz potential 
defined at each point \(x \in \Rset^N \setminus \{0\}\) by 
\begin{align*}
I_\alpha (x) &= \frac{A_\alpha}{\abs{x}^{N - \alpha}}, &
&\text{where }&
&A_\alpha = \frac{\Gamma(\tfrac{N-\alpha}{2})}
                   {\Gamma(\tfrac{\alpha}{2})\pi^{N/2}2^{\alpha} }.
\end{align*}
When \(N = 3\), \(\alpha = 2\) and \(p = 2\), the equation \eqref{eqChoquard} has appeared in several contexts of quantum physics and is known as the \emph{Choquard--Pekar equation} \citelist{\cite{Pekar1954}\cite{Lieb1977}}, the \emph{Schr\"odin\-ger--Newton equation} \cites{KRWJones1995gravitational,KRWJones1995newtonian,Moroz-Penrose-Tod-1998} and the \emph{stationary Hartree equation}.

The action functional \(\mathcal{A}\) associated to the Choquard equation 
\eqref{eqChoquard} is defined for each function \(u\) in the Sobolev space 
\(H^1 (\Rset^N)\) by 
\[
 \mathcal{A} (u) = \frac{1}{2} \int_{\Rset^N} \abs{\nabla u}^2 + \abs{u}^2 - \frac{1}{2 p} \int_{\Rset^N} \bigl(I_\alpha \ast \abs{u}^p\bigr) \abs{u}^p.
\]
In view of the Hardy--Littlewood--Sobolev inequality, which states that if \(s 
\in (1, \frac{N}{\alpha})\) then for 
every \(v \in L^s (\Rset^N)\), \(I_\alpha \ast v\in L^\frac{N s}{N - \alpha s} 
(\Rset^N)\) and
\begin{equation}
\label{eqHLS}
 \int_{\Rset^N} \abs{I_\alpha \ast v}^\frac{N s}{N - \alpha s} \le C 
\Bigl(\int_{\Rset^N} \abs{v}^s \Bigr)^\frac{N}{N - \alpha s},
\end{equation} (see for example \cite{LiebLoss2001}*{theorem 4.3}), and of the 
classical Sobolev embedding, the action functional \(\mathcal{A}\) is 
well-defined and continuously differentiable whenever
\[
 \frac{N - 2}{N + \alpha} \le \frac{1}{p} \le \frac{N}{N + \alpha}.
\]
A natural constraint for the equation is the \emph{Nehari constraint} \(\dualprod{\mathcal{A}' (u)}{u} = 0\) which leads to search for solutions by minimizing the action functional on the \emph{Nehari manifold}
\[
 \mathcal{N}_0 = \bigl\{u \in H^1 (\Rset^N) \setminus \{0\} \st \dualprod{\mathcal{A}' (u)}{u} = 0\}.
\]
The existence of such a solution has been proved when
\[
 \frac{N - 2}{N + \alpha} < \frac{1}{p} < \frac{N}{N + \alpha};
\]
these assumptions are optimal
\citelist{\cite{Lieb1977}\cite{Lions1980}\cite{MorVanSchaft13}}.

We are interested in the construction of \emph{nodal solutions} to 
\eqref{eqChoquard}, that is, solutions to \eqref{eqChoquard} that change sign.
The easiest way to construct such solutions is to impose an odd symmetry 
constraint. More precisely we consider the Sobolev space of odd functions
\begin{multline*}
 H^1_\mathrm{odd}(\Rset^N) 
 = \bigl\{u \in H^1 (\Rset^N)\st \text{for almost every \((x', x_N) \in 
\Rset^N\)},\\
    u(x', -x_N) = - u (x', x_N)\bigr\},
\end{multline*}
we define the odd Nehari manifold
\[
 \mathcal{N}_{\mathrm{odd}}
 = \mathcal{N}_{0} \cap H^1_\mathrm{odd} (\Rset^N)
\]
and the corresponding level 
\[
 c_\mathrm{odd} = \inf_{\mathcal{N}_{\mathrm{odd}}} \mathcal{A}.
\]
Our first result is that this level \(c_\mathrm{odd}\) is achieved.

\begin{theorem}
\label{theoremOdd}
If \(\frac{N - 2}{N + \alpha} < \frac{1}{p}< \frac{N}{N + \alpha}\), then there exists a weak solution \(u \in H^1_\mathrm{odd} (\Rset^N) \cap C^2 (\Rset^N)\) to the Choquard equation \eqref{eqChoquard} such that \(\mathcal{A} (u) = c_\mathrm{odd}\).
Moreover, \(u\) has constant sign on each of the half-spaces \(\Rset^N_+\) and \(\Rset^N_-\) and \(u\) is axially symmetric with respect to an axis perpendicular to \(\partial \Rset^N_+ = \Rset^{N - 1} \times \{0\}\).
\end{theorem}

Nodal solutions with higher level of symmetries and thus larger action have 
already been constructed 
\citelist{\cite{CingolaniClappSecchi2012}\cite{CingolaniClappSecchi2013}\cite{
ClapSalazar2013}}.

The proof of theorem~\ref{theoremOdd} relies on two ingredients: a compactness 
property up to translation under the strict inequality \(c_\mathrm{odd} < 2 
c_0\) obtained by a concentration--compactness argument 
(proposition~\ref{propositionOddPalaisSmaleCondition}) and the proof of the 
latter strict inequality (proposition~\ref{propositionStrictInequality}).

\bigbreak

Another notion of solution is that of \emph{least action nodal solution}, which has been well studied for local problems \citelist{\cite{CeramiSoliminiStruwe1986}\cite{CastroCossioNeuberger1997}\cite{CastorCossioNeuberger1998}}. 
As for these local problems, we define the constrained Nehari nodal set (as in the local case,
in contrast with \(\mathcal{N}_0\) and \(\mathcal{N}_\mathrm{odd}\), 
the set \(\mathcal{N}_{\mathrm{nod}}\) \emph{is not a manifold}),
\begin{multline*}
 \mathcal{N}_{\mathrm{nod}} 
 =\bigl\{ u \in H^1 (\Rset^N) \st u^+ \ne 0 \ne u^-,\,\\
           \dualprod{\mathcal{A}'(u)}{u^+} = 0
           \text{ and } \dualprod{\mathcal{A}'(u)}{u^-} = 0\bigr\},
\end{multline*}
where \(u^+ = \max (u, 0) \ge 0\) and \(u^- = \min (u, 0) \le 0\). (In contrast 
with the local case, we have for every \(u \in \mathcal{N}_{\mathrm{nod}}\), 
\(\dualprod{\mathcal{A}'(u)}{u^+} < \dualprod{\mathcal{A}'(u^+)}{u^+}\), and thus \(u^+ \not \in \mathcal{N}_0\) and \(u^- \not \in \mathcal{N}_0\).)
We prove that when \(p > 2\), the associated level
\[
 c_\mathrm{nod} = \inf_{\mathcal{N}_{\mathrm{nod}}} \mathcal{A}
\]
is achieved.

\begin{theorem}
\label{theoremNod}
If \(\frac{N - 2}{N + \alpha} <  \frac{1}{p} < 
\frac{1}{2}\), then there exists a weak solution \(u \in H^1 (\Rset^N)\cap C^2 (\Rset^N)\) to the 
Choquard equation \eqref{eqChoquard}  such that \(\mathcal{A} (u) = 
c_\mathrm{nod}\), and \(u\) changes sign.
\end{theorem}

The restriction on the exponent \(p\) can only be satisfied when \(\alpha > N - 
4\).
We understand that \(u\) changes sign if the sets \(\{x \in \Rset^N \st u (x) > 
0 \}\) and \(\{x \in \Rset^N \st u(x) < 0 \}\) have both positive measure.

% Nodal solutions which minimize the action among radial solutions were 
% constructed \cite{Ye2015}. In that case the proof can rely strongly on the 
% compactness of radial embeddings.

We do not know whether the solutions  constructed in theorem~\ref{theoremNod} are odd and coincide thus with those of theorem~\ref{theoremOdd} or even whether the solutions of theorem~\ref{theoremNod} have axial symmetry as those of theorem~\ref{theoremOdd}. We leave these questions as open problems.

The proof of theorem~\ref{theoremNod} is based on a new reformulation of the 
minimization problem as a minimax problem that allows to apply a minimax 
principle with location information 
(proposition~\ref{propositionSimpleVariational}) and a new compactness property 
up to translations under the condition \(c_{\mathrm{nod}} < 2 c_0\) proved by 
concentration--compactness 
(proposition~\ref{propositionNodalPalaisSmaleCondition}), in the proof of which 
we introduce suitable methods and estimates (see lemma~\ref{lemmaNewEstimate}).
The latter strict inequality is deduced from the inequality \(c_\mathrm{nod} 
\le c_{\mathrm{odd}}\).

Compared to theorem~\ref{theoremOdd}, theorem~\ref{theoremNod} introduces the additional restriction \(p > 2\). This assumption is almost optimal: in the locally sublinear case \(p < 2\), the level \(c_\mathrm{nod}\) is not achieved.

\begin{theorem}
\label{theoremDegeneracy}
If \(\max (\frac{N - 2}{N + \alpha}, \frac{1}{2}) < \frac{1}{p} < \frac{N}{N 
+ \alpha}\), then \(c_\mathrm{nod} = c_0\) is not achieved in 
\(\mathcal{N}_{\mathrm{nod}}\).
\end{theorem}

Theorem~\ref{theoremDegeneracy} shows that minimizing the action on the Nehari nodal set does not provide a nodal solution; there might however exist a minimal action nodal solution that would be constructed in another fashion.

We do not answer in the present work whether \(c_\mathrm{nod}\) is achieved when 
\(p = 2\) and \(\alpha > N - 4\). In a forthcoming manuscript in collaboration with V.\thinspace{} Moroz, we extend theorem~\ref{theoremDegeneracy} to the case \(p = 2\) by taking the limit \(p \searrow 2\) \cite{GhimentiMorozVanSchaftingen}.

\bigbreak

If we compare the results in the present paper to well-established features of  
the \emph{stationary nonlinear Sch\"odinger equation}
\begin{equation}
\label{NLS}
 -\Delta u + u  = \abs{u}^{2p - 2} u,
\end{equation}
which is the local counterpart of the Choquard equation~\eqref{eqChoquard}, theorems~\ref{theoremOdd} and \ref{theoremNod} are quite surprising.
The action functional associated to \eqref{NLS} is defined by 
\[
 \mathcal{A} (u) 
 = \frac{1}{2} \int_{\Rset^N} \abs{\nabla u}^2 + \abs{u}^2 
   - \frac{1}{2 p} \int_{\Rset^N} \abs{u}^{2 p},
\]
which is well-defined and continuously differentiable when \(\frac{1}{2} - \frac{1}{N} < \frac{1}{p} < \frac{1}{2}\).
Since in this case \(\mathcal{A} (u) = \mathcal{A} (u^+) + \mathcal{A} (u^-)\), 
it can be easily proved by a density argument that 
\[
 c_\mathrm{odd} = c_{\mathrm{nod}} = 2 c_0.
\]
Therefore if one of the infimums \( c_\mathrm{odd} \) or \(c_{\mathrm{nod}}\) is 
achieved at \(u\), then both \(u^+\) and \(u^-\) should achieve \(c_0\) in 
\(\mathcal{N}_0\). This is impossible, since by the strong maximum principle 
\(u^+> 0\) and \(u^- > 0\) almost everywhere on the space \(\Rset^N\).
This nonexistence of minimal action nodal solutions also contrasts with 
theorem~\ref{theoremDegeneracy}: for the nonlocal problem \(c_{\mathrm{nod}}\) is too 
small to be achieved whereas for the local one this level is too large.

\section{Minimal action odd solution}

In this section we prove theorem~\ref{theoremOdd} about the existence of solutions under an oddness constraint. 

\subsection{Variational principle}

We first observe that the corresponding level \(c_{\mathrm{odd}}\) is positive.
\begin{proposition}[Nondegeneracy of the level]
\label{propositionOddNondegenerate}
If  \(\frac{N - 2}{N + \alpha} \le \frac{1}{p} \le \frac{N}{N + \alpha}\), then 
\[
  c_{\mathrm{odd}} > 0.
\]
\end{proposition}

\begin{proof}
Since \(\mathcal{N}_{\mathrm{odd}} = \mathcal{N}_0 \cap 
H^1_{\mathrm{odd}} 
(\Rset^N) \subset \mathcal{N}_0\) we have \(c_{\mathrm{odd}} \ge c_0\). The 
conclusion follows then from the fact that \(c_0 > 0\) \cite{MorVanSchaft13}.
\end{proof}
% \begin{proof}
% This follows immediately from the fact that the corresponding smallerquantity 
% \(c_0\) for the problem without any symmetry constraint is already positive 
% \cite{}. This idea gives the following direct argument. By the 
% Hardy--Littewood--Sobolev inequality and the classical Sobolev inequality we 
% have for every \(u \in H^1 (\Rset^N)\), if \(\dualprod{\mathcal{A}' 
% (u)}{u}=0\) 
% then 
% \[
%  \int_{\Rset^N} \abs{\nabla u}^2 + \abs{u}^2 
%  \le \int_{\Rset^N} \bigl(I_{\alpha} \ast \abs{u}^p\bigr) \abs{u}^p \le 
%      C \Bigl(\int_{\Rset^N} \abs{\nabla u}^2 + \abs{u}^2 \Bigr)^{p},
% \]
% and thus 
% \[
%   \mathcal{A} (u) = \frac{1}{2} \int_{\Rset^N} \abs{\nabla u}^2 + \abs{u}^2 
%   - \frac{1}{2p} \int_{\Rset^N} \bigl(I_{\alpha} \ast \abs{u}^p\bigr) 
% \abs{u}^p
%   = \Bigl(\frac{1}{2} - \frac{1}{2 p}\Bigr)\abs{\nabla u}^2 + \abs{u}^2 
%   \le c,
% \]
% for some constant \(c > 0\). Therefore, we reach the conclusion.
% \end{proof}

A first step in the construction of our solution is the existence
a Palais--Smale sequence.

\begin{proposition}[Existence of a Palais-Smale sequence]
\label{propositionOddPalaisSmaleExistence}
If \(\frac{N - 2}{N + \alpha} \le \frac{1}{p} \le \frac{N}{N + \alpha}\), then 
there exists a sequence \((u_n)_{n \in \N}\) in \(H^1_{\mathrm{odd}} 
(\Rset^N)\) such that, as \(n \to \infty\),
\begin{align*}
  \mathcal{A} (u_n) & \to c_{\mathrm{odd}} &
  &\text{ and} &
  \mathcal{A}' (u_n) & \to 0 \quad \text{in \(H^1_{\mathrm{odd}} (\Rset^N)'\)}.
\end{align*}
\end{proposition}

\begin{proof}
We first recall that the level \(c_{\mathrm{odd}}\) can be rewritten as a mountain pass
minimax level:
\[
  c_{\mathrm{odd}} 
  = \inf_{\gamma \in \Gamma} \sup_{[0, 1]} \mathcal \mathcal{A} \circ \gamma,
\]
where the class of paths \(\Gamma\) is defined by 
\[
 \Gamma = \bigl\{ \gamma \in C \bigl([0, 1], H^1_\mathrm{odd} (\Rset^N)\bigr) \st \gamma (0)= 
0 \text{ and } \mathcal{A} \bigl(\gamma (1)\bigr) < 0\bigr\}
\]
(see for example \cite{Willem1996}*{theorem 4.2}).
By the general minimax principle 
\cite{Willem1996}*{theorem 2.8}, there exists a sequence \((u_n)_{n \in \N}\) 
in \(H^1_{\mathrm{odd}} (\Rset^N)\) such that the sequence
\(\bigl(\mathcal{A} (u_n)\bigr)_{n \in \N}\) converges to \(c_\mathrm{odd}\) 
and the 
sequence
\(\bigl(\mathcal{A}' (u_n)\bigr)_{n \in \N}\) converges strongly to \(0\) in 
the dual space \( H^1_{\mathrm{odd}} (\Rset^N) '\).
\end{proof}

\subsection{Palais--Smale condition}
We would now like to construct out of the Palais--Smale sequence of 
proposition~\ref{propositionOddPalaisSmaleExistence} a solution to our problem.
We shall prove that the functional \(\mathcal{A} \vert_{H^1_{\mathrm{odd}} 
(\Rset^N)}\) satisfies the Palais--Smale condition up to translations at the 
level \(c_{\mathrm{odd}}\) if the strict inequality \(c_\mathrm{odd} < 2 c_0\) 
holds.

\begin{proposition}[Palais--Smale condition]
\label{propositionOddPalaisSmaleCondition}
Assume that \(\frac{N - 2}{N + \alpha} < \frac{1}{p} < \frac{N}{N + 
\alpha}\).
Let \((u_n)_{n \in \N}\) be a sequence in \(H^1_{\mathrm{odd}} (\Rset^N)\) such 
that, as \(n \to \infty\),
\begin{align*}
  \mathcal{A} (u_n) & \to c_{\mathrm{odd}} &
  &\text{ and} &
  \mathcal{A}' (u_n) & \to 0 \quad \text{in \(H^1_{\mathrm{odd}} (\Rset^N)'\)}.
\end{align*}
If
\[
  c_{\mathrm{odd}} < 2 c_0,
\]
then there exists a sequence of points \((a_n)_{n \in \N}\) in \(\Rset^{N - 1} \times \{0\}\subset \Rset^N\) such 
that the subsequence
\((u_{n_k} \bigl(\cdot - a_{n_k})\bigr)_{k \in \N}\) converges strongly in 
\(H^1 
(\Rset^N)\) to \(u \in H^1_{\mathrm{odd}} (\Rset^N)\).
Moreover
\begin{align*}
  \mathcal{A} (u) &= c_{\mathrm{odd}} &
  &\text{ and }&
  \mathcal{A}'(u) &= 0 \quad \text{in \(H^1 (\Rset^N)'\)}.
\end{align*}
\end{proposition}

\begin{proof}
First, we observe that, as \(n \to \infty\),
\begin{equation}
\label{eqNorm}
\begin{split}
  \Bigl(\frac{1}{2} - \frac{1}{2 p}\Bigr)\int_{\Rset^N} \abs{\nabla u_n}^2 + 
\abs{u_n}^2 &= \mathcal{A} (u_n) - \frac{1}{2 p} \dualprod{\mathcal{A}' 
(u_n)}{u_n}\\
  &= \mathcal{A} (u_n) + o \biggl( \Bigl(\int_{\Rset^N} \abs{\nabla u_n}^2 + 
\abs{u_n}^2 \Bigr)^\frac{1}{2}\biggr)\\
& = c_{\mathrm{odd}} + o \biggl( \Bigl(\int_{\Rset^N} \abs{\nabla u_n}^2 + 
\abs{u_n}^2 \Bigr)^\frac{1}{2}\biggr).
\end{split}
\end{equation}
In particular, the sequence \((u_n)_{n \in \N}\) is bounded in the space \(H^1 (\Rset^N)\).
% Therefore, we have \(\lim_{n \to \infty} \dualprod{\mathcal{A}' (u_n)}{u_n} = 
% 0\) and thus, in view of proposition~\ref{propositionOddNondegenerate},
% % \[
% %   \lim_{n \to \infty} \Bigl(\frac{1}{2} - \frac{1}{2 p}\Bigr) 
% \int_{\Rset^N} % \abs{\nabla u_n}^2 + \abs{u_n}^2 = \lim_{n \to \infty} 
% \mathcal{A} (u_n) - % \frac{1}{2 p}\dualprod{\mathcal{A}' (u_n)}{u_n} 
% =c_{\mathrm{odd}} > 0.
% % \]
% % We also have 
% \[
%  \lim_{n \to \infty} \Bigl(\frac{1}{2}-\frac{1}{2 p} \Bigr)\int_{\Rset^N} 
% \bigl(I_\alpha \ast \abs{u_n}^p\bigr) \abs{u_n}^p = \lim_{n \to \infty} 
% \mathcal{A} (u_n) - \frac{1}{2} \dualprod{\mathcal{A}' (u_n)}{u_n} = 
% c_{\mathrm{nod}} > 0.
% \]
% By the Hardy--Littlewood--Sobolev inequality \eqref{eqHLS}, we deduce that 
% \[
%   C \liminf_{n \to \infty} \Bigl(\int_{\Rset^N} \abs{u_n}^\frac{2 N p}{N + 
% \alpha}\Bigr)^{\frac{N + \alpha}{N}}
%   \ge \lim_{n \to \infty}\int_{\Rset^N} \bigl(I_\alpha \ast \abs{u_n}^p\bigr) 
% \abs{u_n}^p > 0.
% \]

We now claim that there exists \(R > 0\) such that 
\begin{equation}
\label{eqDR}
  \liminf_{n \to \infty} \int_{D_R} \abs{u_n}^\frac{2 N p}{N + \alpha} > 0,
\end{equation}
where the set \(D_R \subset \Rset^N\) is the infinite slab
\[
 D_R = \Rset^{N - 1} \times [-R, R].
\]
We assume by contradiction that for each \(R > 0\),
\[
 \liminf_{n \to \infty} \int_{D_R} \abs{u_n}^\frac{2 N p}{N + \alpha} = 0.
\]
We define for each \(n \in \N\) the functions \(v_n = \chi_{\Rset^{N - 1} 
\times (0, \infty)} u_n\) and \(\Tilde{v}_n = \chi_{\Rset^{N - 1} \times 
(-\infty, 0)}u_n\). 
Since \(u_n \in H^1_{\mathrm{odd}}(\Rset^N)\), we have \(v_n \in H^1_0 (\Rset^{N - 1} 
\times (0, \infty))\subset H^1(\Rset^N)\) and \(\Tilde{v}_n  \in H^1_0 (\Rset^{N - 1}
\times (-\infty, 0)) \subset H^1(\Rset^N)\).
We now compute
\[
\begin{split}
  \int_{\Rset^N} \bigl(I_\alpha \ast \abs{v_n}^p\bigr) \abs{\Tilde{v}_n}^p
  \le &2 \int_{\Rset^N} \int_{D_R} I_\alpha (x - y) \abs{v_n (y)}^p 
\abs{\Tilde{v}_n (x)}^p \dif y\dif x\\
  &+\int_{\Rset^N \setminus D_R} \int_{\Rset^N \setminus D_R} I_\alpha (x - y) 
\abs{v_n (y)}^p \abs{\Tilde{v}_n (x)}^p \dif y\dif x\\
\end{split}
\]
By definition of the region \(D_R\) we have, if 
\(\beta \in (\alpha, N)\),
\begin{multline*}
  \int_{\Rset^N} \bigl(I_\alpha \ast \abs{v_n}^p\bigr) \abs{\Tilde{v}_n}^p\\
  \le 2 \int_{D_R} \bigl( I_\alpha \ast  
 \abs{u_n}^p\bigr) \abs{u_n}^p +\int_{\Rset^N} \bigl((\chi_{\Rset^N \setminus 
B_{2R}} I_\alpha) \ast  \abs{u_n}^p\bigr) \abs{u_n}^p\\
\le 2 \int_{D_R} \bigl( I_\alpha \ast  
 \abs{u_n}^p\bigr) \abs{u_n}^p +\frac{C}{R^{\beta - \alpha}} \int_{\Rset^N} 
\bigl((\chi_{\Rset^N \setminus B_{2R}} 
I_{\beta}) \ast  \abs{u_n}^p\bigr) \abs{u_n}^p
\end{multline*}
Since by assumption \(p > \frac{N + \alpha}{N}\), we can take \(\beta\) such that moreover \(\beta < (p - 1) N\), and then by the 
Hardy--Littlewood--Sobolev inequality \eqref{eqHLS} and the classical Sobolev 
inequality, we 
obtain that 
\begin{multline*}
 \int_{\Rset^N} \bigl(I_\alpha \ast \abs{v_n}^p\bigr) \abs{\Tilde{v}_n}^p
  \le C' \Bigl(\int_{\Rset^N} \abs{\nabla u_n}^2 + 
\abs{u_n}^2\Bigr)^\frac{p}{2} \Bigl(\int_{D_R} \abs{u_n}^\frac{2 N p}{N + 
\alpha}\Bigr)^\frac{N + \alpha}{2 N}\\ + 
\frac{C''}{R^{\beta - \alpha}} \Bigl(\int_{\Rset^N} \abs{\nabla u_n}^2 + 
\abs{u_n}^2\Bigr)^p,
\end{multline*}
from which we deduce that
\begin{equation}
\label{eqCrossIntegralVanishing}
  \lim_{n \to \infty}  
    \int_{\Rset^N} \bigl(I_\alpha \ast \abs{v_n}^p\bigr)
                 \abs{\Tilde{v}_n}^p 
   = 0.
\end{equation}
For each \(n \in \N\), we fix  \(t_n \in (0, \infty)\) so that 
\(t_n v_n \in \mathcal{N}_0\) or, equivalently, 
\begin{equation}
\label{eqTaun}
\begin{split}
  t_n^{2 p - 2} &= \frac{\displaystyle \int_{\Rset^N} \abs{\nabla v_n}^2 + 
\abs{v_n}^2}{\displaystyle \int_{\Rset^N} \bigl(I_\alpha \ast \abs{v_n}^p\bigr) 
\abs{v_n}^p}\\
  &= \frac{\displaystyle \int_{\Rset^N} \abs{\nabla u_n}^2 + 
\abs{u_n}^2}{\displaystyle \int_{\Rset^N} \bigl(I_\alpha \ast \abs{u_n}^p\bigr) 
\abs{u_n}^p - 2 \int_{\Rset^N} \bigl(I_\alpha \ast \abs{v_n}^p\bigr) 
\abs{\Tilde{v}_n}^p}.
\end{split}
\end{equation}
For every \(n \in \N\), we have 
\[
 \mathcal{A} (t_n u_n) = 2 \mathcal{A} (t_n v_n) 
- \frac{t_n^{2 p}}{p} \int_{\Rset^N} \bigl(I_\alpha \ast \abs{v_n}^p\bigr) 
\abs{\Tilde{v}_n}^p
\]
By \eqref{eqNorm}, \eqref{eqCrossIntegralVanishing} and \eqref{eqTaun}, in view of proposition~\ref{propositionOddNondegenerate}, we note that \(\lim_{n \to \infty} t_n = 1\) and thus in view 
of \eqref{eqCrossIntegralVanishing} again we conclude that
\[
  c_{\mathrm{odd}} = \lim_{n \to \infty} \mathcal{A} (u_n)  = \lim_{n \to 
\infty} \mathcal{A} (t_n u_n) = 2 \lim_{n \to \infty} \mathcal{A} (t_n 
v_n) \ge 2 c_0,
\]
in contradiction with the assumption \(c_{\mathrm{odd}} < 2 c_0\) of the 
proposition. 

We can now fix \(R > 0\) such that \eqref{eqDR} holds. We take a function \(\eta \in 
C^\infty (\Rset^N)\) such that \(\supp \eta \subset D_{3 R /2}\),  \(\eta = 1\) 
on \(D_R\), \(\eta \le 1\) on \(\Rset^N\) and \(\nabla 
\eta \in L^\infty (\Rset^N)\). 
We have the inequality \citelist{\cite{Lions1984CC2}*{lemma I.1}\cite{Willem1996}*{lemma 
1.21}\cite{MorVanSchaft13}*{lemma 2.3}\cite{VanSchaftingen2014}*{(2.4)}}
\[
\begin{split}
  \int_{D_R} \abs{u_n}^{\frac{2 N p}{N + \alpha}} &\le \int_{\Rset^N} \abs{\eta 
u_n}^{\frac{2 N p}{N + \alpha}}\\
  &\le C \Bigl(\sup_{a \in \Rset^N} \int_{B_{R/2} (a)} \abs{\eta u_n}^{\frac{2 
N p}{N + \alpha}} \Bigr)^{1 - \frac{N + \alpha}{N p}}
   \int_{\Rset^N} \abs{\nabla (\eta u_n)}^2 + \abs{\eta u_n}^2\\
  &\le C' \Bigl(\sup_{a \in \Rset^{N - 1} \times \{0\}} \int_{B_{2R} (a)} 
\abs{u_n}^{\frac{2 N p}{N + \alpha}} \Bigr)^{1 - \frac{N + \alpha}{N p}} 
\int_{\Rset^N} \abs{\nabla u_n}^2 + \abs{u_n}^2.
 \end{split}
\]
Since the sequence \((u_n)_{n \in \N}\) is bounded in the space \(H^1 (\Rset^N)\) we 
deduce from \eqref{eqDR} that there exists
a sequence of points \((a_n)_{n \in \N}\) in the hyperplane \(\Rset^{N - 1} \times \{0\}\) such that 
\[
  \liminf_{n \to \infty} \int_{B_{2 R} (a_n)} \abs{u_n}^\frac{2 N p}{N + 
\alpha} > 0.
\]
Up to translations and a subsequence, we can assume that the sequence 
\((u_n)_{n \in \N}\) converges weakly in \(H^1 
(\Rset^N)\) to a function \(u \in H^1 
(\Rset^N)\). 

Since the action functional \(\mathcal{A}\) is invariant under odd reflections, 
we note that for every \(n \in \N\), \(\mathcal{A} (u_n) = 0\) on 
\(H^1_{\mathrm{odd}} (\Rset^N)^\perp\) by the symmetric criticality principle 
\cite{Palais1979} (see also~\cite{Willem1996}*{theorem 1.28}). This allows to 
deduce from the strong convergence of the sequence \((\mathcal{A}' (u_n))_{n 
\in \N}\) to \(0\) in \(H^1_{\mathrm{odd}} (\Rset^N)'\) the strong 
convergence to \(0\) of the sequence \((\mathcal{A}' (u_n))_{n \in \N}\) in 
\(H^1 (\Rset^N)'\).

For any test function \(\varphi \in C^1_c (\Rset^N)\), by the weak 
convergence of the sequence \((u_n)_{n \in \N}\), we first have
\[
  \lim_{n \to \infty} \int_{\Rset^N} \scalprod{\nabla u_n}{\nabla \varphi} + 
u_n \varphi = \int_{\Rset^N} \scalprod{\nabla u}{\nabla \varphi} + u \varphi.
\]
By the classical Rellich--Kondrashov compactness theorem, the sequence 
\((\abs{u_n}^p)_{n \in \N}\) converges locally in measure to \(\abs{u}^p\) and by the 
Sobolev inequality, this sequence is bounded in \(L^\frac{2 N}{N + \alpha} 
(\Rset^N)\). Therefore, the sequence \((\abs{u_n}^p)_{n \in \N}\) converges 
weakly to \(\abs{u}^p\) in \(L^\frac{2 N}{N + \alpha} (\Rset^N)\) (see for 
example 
\citelist{\cite{Bogachev2007}*{proposition~4.7.12}\cite{Willem2013}*{
proposition 5.4.7}}), and, by the Hardy--Littlewood--Sobolev inequality \eqref{eqHLS}, the 
sequence \((I_\alpha \ast \abs{u_n}^p)_{n \in \N}\) converges weakly in 
\(L^{\frac{2 N}{N - \alpha}} (\Rset^N)\) to \(I_\alpha \ast \abs{u}^p\).
By the Rellich--Kondrashov theorem again, the sequence \(((I_\alpha \ast 
\abs{u_n}^p)\abs{u_n}^{p - 2} 
u_n)_{n \in \N}\) converges weakly in \(L^\frac{2 N}{N + 2} (K)\) for every 
compact set \(K \subset \Rset^N\).
Therefore we have
\[
  \lim_{n \to \infty} \int_{\Rset^N} (I_\alpha \ast \abs{u_n}^p)\abs{u_n}^{p - 
2} u_n \varphi = \int_{\Rset^N} \bigl(I_\alpha \ast \abs{u}^p\bigr)\abs{u}^{p - 2} u 
\varphi. 
\]
We have thus proved that  
\[
   \mathcal{A}' (u) = 0 = \lim_{n \to \infty} \mathcal{A}' (u_n).
\]
Finally, we have 
\[
\begin{split}
\lim_{n \to \infty} \mathcal{A} (u_n) &= \lim_{n \to \infty} \mathcal{A} (u_n) 
- \frac{1}{2 p} \dualprod{\mathcal{A}' (u_n)}{u_n}\\
& =\lim_{n \to \infty} \Bigl(\frac{1}{2} - \frac{1}{2 p} \Bigr)\int_{\Rset^N} 
\abs{\nabla u_n}^2 + \abs{u_n}^2\\
&\ge \Bigl(\frac{1}{2} - \frac{1}{2 p} \Bigr)\int_{\Rset^N} \abs{\nabla u}^2 + 
\abs{u}^2\\
&=\mathcal{A} (u) - \frac{1}{2 p} \dualprod{\mathcal{A}' (u)}{u} = \mathcal{A} (u),
\end{split}
\]
from which we conclude that \(\mathcal{A} (u) = c_{\mathrm{odd}}\) and that 
the sequence \((u_n)_{n \in \N}\) converges strongly to \(u\) in \(H^1 
(\Rset^N)\).
\end{proof}

\subsection{Strict inequality}

It remains now to establish the strict inequality \(c_\mathrm{odd} < 2 c_0\).

\begin{proposition}
\label{propositionStrictInequality}
If \(\frac{N - 2}{N + \alpha} < \frac{1}{p} < \frac{N}{N + \alpha}\), then 
\[
  c_{\mathrm{odd}} < 2 c_0.
\]
\end{proposition}

\begin{proof}
It is known that the Choquard equation has a least action solution 
\cite{MorVanSchaft13}.
More precisely, there exists \(v \in H^1 (\Rset^N) \setminus \{0\}\) such that 
\(\mathcal{A}' (v)= 0\)
and 
\[
 \mathcal{A}(v) = \inf_{\mathcal{N}_0} \mathcal{A}.
\]

We take a function \(\eta \in C^2_c (\Rset^N)\) such that \(\eta = 1\) on 
\(B_1\), \(0 \le \eta \le 1\) on \(\Rset^N\) and \(\supp \eta \subset B_2\) and 
we define 
for each \(R > 0\) the function \(\eta_R \in C^2_c (\Rset^N)\) for every 
\(x \in \Rset^N\) by 
\(\eta_R (x) = \eta (x/R)\).
We define now the function \(u_R : \Rset^N \to \Rset\) for each \(x = (x', x_N) 
\in \Rset^N\) by
\[
  u_{R} (x) =(\eta_R  v) (x', x_N - 2R) - (\eta_R  v) (x', - x_N - 2R).
\]
It is clear that \(u_R \in H^1_{\mathrm{odd}} (\Rset^N)\).

We observe that \(\dualprod{\mathcal{A}'(t_R u_R)}{t_R u_R}=0\) if and only if 
\(t_R \in (0, \infty)\) satisfies
\[
{\displaystyle t_{R}^{2p-2}=\frac{\displaystyle \int_{\Rset^{N}}\abs{\nabla 
u_{R}}^{2}+\abs{u_{R}}^{2}}{\displaystyle 
\int_{\Rset^{N}}\bigl(I_{\alpha}\ast\abs{ u_ { R } } ^ { p } 
\bigr)\abs{u_{R}}^{p}}}.
\]
Such a \(t_R\) always exists and 
\[
 \mathcal{A} (t_R u_R) = 
 \Bigl(\frac{1}{2} - \frac{1}{p}\Bigr) \frac{\displaystyle \Bigl(\int_{\Rset^N} 
\abs{\nabla u_R}^2 + \abs{u_R}^2 \Bigr)^\frac{p}{p - 1}}
 {\displaystyle \Bigl(\int_{\Rset^N} (I_\alpha \ast \abs{u_R}^p) \abs{u_R}^p 
\Bigr)^\frac{1}{p - 1}}.
\]
The proposition will follow once we have established that for some \(R >0\)
\begin{equation}
\label{eqReformulatedStrict}
 \frac{\displaystyle \Bigl(\int_{\Rset^N} \abs{\nabla u_R}^2 + \abs{u_R}^2 
\Bigr)^\frac{p}{p - 1}}
 {\displaystyle \Bigl(\int_{\Rset^N} (I_\alpha \ast \abs{u_R}^p) \abs{u_R}^p 
\Bigr)^\frac{1}{p - 1}}
 < 2 \frac{\displaystyle \Bigl(\int_{\Rset^N} \abs{\nabla v}^2 + \abs{v}^2 
\Bigr)^\frac{p}{p - 1}}
 {\displaystyle \Bigl(\int_{\Rset^N} (I_\alpha \ast \abs{v}^p) \abs{v}^p 
\Bigr)^\frac{1}{p - 1}}.
\end{equation}

We begin by estimating the denominator in the 
left-hand side of \eqref{eqReformulatedStrict}. We 
first observe that, by construction of the function \(u_R\)
\[
\int_{\Rset^{N}}\bigl(I_{\alpha}\ast\abs{u_{R}}^{p}\bigr)\abs{u_{R}}^{p}\ge 
2\int_{\Rset^{N}}\bigl(I_{\alpha}\ast\abs{\eta_{R}v}^{p}\bigr)\abs{\eta_{R}v}^{p
}+ 2\frac{A_\alpha}{(4R)^{N - \alpha}} \Bigl(\int_{\Rset^N} \abs{\eta_R v 
}^p\Bigr)^2.
\]
For the first term, we have 
\[
\begin{split}
\int_{\Rset^{N}}\bigl(I_{\alpha}\ast\abs{\eta_{R}v}^{p}\bigr)\abs{\eta_{R}v}^{p}
&=\int_{\Rset^{N}}\bigl(I_{\alpha}\ast\abs{v}^{p}\bigr)\abs{v}^{p} - 
2\int_{\Rset^{N}}\bigl(I_{\alpha}\ast\abs{v}^{p}\bigr)(1 - 
\eta_R^p)\abs{v}^{p}\\
&\qquad \qquad + \int_{\Rset^{N}}\bigl(I_{\alpha}\ast(1 - 
\eta_R^p)\abs{v}^{p}\bigr)(1 - \eta_R^p)\abs{v}^{p}\\
&\ge \int_{\Rset^{N}}\bigl(I_{\alpha}\ast\abs{v}^{p}\bigr)\abs{v}^{p} - 2
\int_{\Rset^{N}}\bigl(I_{\alpha}\ast \abs{v}^{p}\bigr)(1 - \eta_R^p) 
\abs{v}^{p}.
\end{split}
\]
By the asymptotic properties of \(I_\alpha \ast 
\abs{v}^p\) \cite{MorVanSchaft13}*{theorem 4}, we have 
\[
 \lim_{\abs{x} \to \infty} \frac{\bigl(I_{\alpha}\ast 
\abs{v}^{p}\bigr)}{I_{\alpha} (x)} = \int_{\Rset^N} \abs{v}^p,
\]
so that 
\[
 2\int_{\Rset^{N}}\bigl(I_{\alpha}\ast \abs{v}^{p}\bigr)(1 - \eta_R^p) 
\abs{v}^{p}
 \le C \int_{\Rset^N \setminus B_R} \frac{\abs{v (x)}^p}{\abs{x}^{N - 
\alpha}}\dif x.
\]
We have thus 
\begin{multline*}
 \int_{\Rset^{N}}\bigl(I_{\alpha}\ast\abs{u_{R}}^{p}\bigr)\abs{u_{R}}^{p}\\
 \ge 2 \int_{\Rset^N} 
\bigl(I_{\alpha}\ast\abs{v}^{p}\bigr)\abs{v}^{p}
 + \frac{2 A_\alpha}{(4R)^{N - \alpha}} \Bigl(\int_{B_R} \abs{v }^p\Bigr)^2
 - C \int_{\Rset^N \setminus B_R} \frac{\abs{v (x)}^p}{\abs{x}^{N - 
\alpha}}\dif x.
\end{multline*}
We now use the information that we have on the decay of the least action 
solution \(v\) \cite{MorVanSchaft13}.
If \(p < 2\), then \(v (x) = O (\abs{x}^{-(N - \alpha)/(2 - p)})\) as \(\abs{x} 
\to \infty\) and 
\[
  \int_{\Rset^N \setminus B_R}  \frac{\abs{v (x)}^p}{\abs{x}^{N - 
\alpha}}\dif x = O \biggl(\frac{1}{R^\frac{Np 
- 2 \alpha}{2 - p}} \biggr) = o \Bigl(\frac{1}{R^{N - \alpha}} \Bigr),
\]
since \(p > \frac{N + \alpha}{N} >  \frac{2N}{2N - \alpha}\).
If \(p \ge 2\), then \(v\) decays exponentially at infinity. We have thus the 
asymptotic lower bound
\begin{multline}
\label{eqDenominator}
 \int_{\Rset^{N}}\bigl(I_{\alpha}\ast\abs{u_{R}}^{p}\bigr)\abs{u_{R}}^{p}\\ \ge 
2 \int_{\Rset^N} 
\bigl(I_{\alpha}\ast\abs{v}^{p}\bigr)\abs{v}^{p} + \frac{2 
A_\alpha}{(4R)^{N - \alpha}} \Bigl(\int_{\Rset^N} \abs{v }^p\Bigr)^2 + o 
\Bigl(\frac{1}{R^{N - \alpha}}\Bigr).
\end{multline}

For the numerator in \eqref{eqReformulatedStrict}, we compute by integration 
by parts
\[
\begin{split}
  \int_{\Rset^N} \abs{\nabla u_R}^2 + \abs{u_R}^2 
  & = 2 \int_{\Rset^N} \abs{\nabla (\eta_R v)}^2 + \abs{\eta_R v}^2 \\
  & = 2 \int_{\Rset^N} \eta_R^2 (\abs{\nabla v}^2 + \abs{v}^2) 
     - 2\int_{\Rset^N} \eta_R (\Delta \eta_R) \abs{v}^2\\
  & \le 2 \int_{\Rset^N} \bigl(\abs{\nabla v}^2 + \abs{v}^2\bigr) + \frac{C}{R^2} 
\int_{B_{2R} \setminus B_R} \abs{v}^2.
\end{split}
\]
If \(p < 2\), we have by the decay of the solution \(v\)
\[
 \frac{1}{R^2} \int_{B_{2R} \setminus B_R} \abs{v}^2
 = O \biggl(\frac{1}{R^{{\frac{Np - 2\alpha}{2 - p} + 2}}} \biggr) 
 = o \Bigl(\frac{1}{R^{N - \alpha}} \Bigr).
\]
In the case where \(p \ge 2\), the solution \(v\) decays exponentially. We 
conclude thus that 
\begin{equation}
\label{eqNumerator}
  \int_{\Rset^N} \abs{\nabla u_R}^2 + \abs{u_R}^2 
  =2 \int_{\Rset^N} \abs{\nabla v}^2 + \abs{v}^2 + o \Bigl(\frac{1}{R^{N - 
\alpha}} \Bigr).
\end{equation}

We derive from the asymptotic bounds \eqref{eqDenominator} and 
\eqref{eqNumerator}, an asymptotic bound on the quotient:
\begin{multline*}
 \frac{\displaystyle \Bigl(\int_{\Rset^N} \abs{\nabla u_R}^2 + \abs{u_R}^2 
\Bigr)^\frac{p}{p - 1}}
      {\displaystyle \Bigl(\int_{\Rset^N} (I_\alpha \ast \abs{u_R}^p) 
\abs{u_R}^p \Bigr)^\frac{1}{p - 1}}\\
 \le 2 \frac{\displaystyle \Bigl(\int_{\Rset^N} \abs{\nabla v}^2 + \abs{v}^2 
\Bigr)^\frac{p}{p - 1}}
 {\displaystyle \Bigl(\int_{\Rset^N} (I_\alpha \ast \abs{v}^p) \abs{v}^p 
\Bigr)^\frac{1}{p - 1}}
 \Biggl(1 
   - \frac{p A_\alpha\displaystyle \Bigl(\int_{B_R} \abs{v}^p\Bigr)^2}
          {(p - 1)(4R)^{N - \alpha}\displaystyle 
              \int_{\Rset^N} (I_\alpha \ast \abs{v}^p) \abs{v}^p}\\ 
              + o \Bigl(\frac{1}{R^{N - \alpha}}\Bigr)\Biggr).
\end{multline*}
The inequality \eqref{eqReformulatedStrict} holds thus when \(R\) is large 
enough, and the conclusion follows.
\end{proof}

\subsection{Existence of a minimal action odd solution}

We have now developped all the tools to prove the existence of a least action 
odd solution to the Choquard equation, corresponding to the existence part of theorem~\ref{theoremOdd}.

\begin{proposition}
\label{propositionExistence}
If \(\frac{N - 2}{N + \alpha} < \frac{1}{p}< \frac{N}{N + \alpha}\), then there exists solution \(u \in H^1_\mathrm{odd} (\Rset^N) \cap C^2 (\Rset^N)\) to the Choquard equation \eqref{eqChoquard} such that \(\mathcal{A} (u) = c_\mathrm{odd}\).
\end{proposition}

\begin{proof}
Let \((u_n)_{n \in \N}\) be the sequence given by 
proposition~\ref{propositionOddPalaisSmaleExistence}. In view of 
proposition~\ref{propositionStrictInequality}, 
proposition~\ref{propositionOddPalaisSmaleCondition} is applicable and gives 
the required weak solution \(u \in H^1 (\Rset^N)\).
By the regularity theory for the Choquard equation \cite{MorVanSchaft13}*{proposition 4.1} (see also \cite{CingolaniClappSecchi2012}*{lemma A.10}), \(u \in C^2 (\Rset^N)\).
\end{proof}

\subsection{Sign and symmetry properties}
We complete the proof of theorem~\ref{theoremOdd} by showing that such solutions have a simple sign and symmetry structure.

\begin{proposition}
\label{propositionQualitative}
If \(\frac{N - 2}{N + \alpha} <  \frac{1}{p} <
\frac{1}{2}\) and if \(u \in H^1_\mathrm{odd} (\Rset^N)\) is a solution to the Choquard equation \eqref{eqChoquard} such that \(\mathcal{A} (u) = c_\mathrm{odd}\),
then \(u\) has constant sign on \(\Rset^N_+\) and \(u\) is axially symmetric with respect to an axis perpendicular to \(\partial \Rset^N_+\).
\end{proposition}
%TODO: Regularity
The proof takes profit of the structure of the problem to rewrite it as a groundstate
of a problem on the halfspace where quite fortunately the strategy for proving similar properties of groundstates of the Choquard equation still works \cite{MorVanSchaft13}*{Propositions 5.1 and 5.2} (see also \cite{MorVanSchaft15}*{propositions 5.2 and 5.3}.

\begin{proof}[Proof of proposition~\ref{propositionQualitative}]
We first rewrite the problem of finding odd solutions on \(\Rset^N\) as a groundstate problem on \(\Rset^N_+\) whose nonlocal term has a more intricate structure.

\begin{claim}
\label{claimRewrite}
For every \(v \in H^1_{\mathrm{odd}} (\Rset^N)\),
\[
  \mathcal{A} (v) = \Tilde{\mathcal{A}} (v\vert_{\Rset^N_+}),
\]
where \(\Rset^N_+ = \Rset^{N - 1} \times (0, \infty)\) and the functional \(\Tilde{\mathcal{A}} : H^1_0 (\Rset^N_+) \mapsto \Rset\) is defined for \(w \in H^1_0 (\Rset^N_+)\) by 
\[
  \Tilde{\mathcal{A}} (w) = \int_{\Rset^N_+} \abs{\nabla w}^2 + \abs{w}^2 - 
  \frac{1}{p} \int_{\Rset^N_+} \int_{\Rset^N_+} K (\abs{x' - y'},x_N, y_N)
  \abs{u (x)}^p \abs{u (y)}^p \dif x \dif y,
\]
with \(x= (x', x_N)\), \(y = (y', y_N)\) and the kernel \(K : (0, \infty)^3 \to \Rset\)
defined for each \((r, s,t) \in (0, \infty)^3\) by 
\[
 K (r, s, t) = \frac{A_\alpha}{\bigl(r^2 + (s - t)^2\bigr)^\frac{N - \alpha}{2}}
 + \frac{A_\alpha}{\bigl(r^2 + (s + t)^2\bigr)^\frac{N - \alpha}{2}}.
\]
In particular, \(u \in \Tilde{\mathcal{N}}_{\mathrm{nod}}\), where 
\[
  \Tilde{\mathcal{N}}_{\mathrm{nod}}
  =\bigl\{ w \in H^1_0 (\Rset^N_+) \st \dualprod{\Tilde{\mathcal{A}}' (w)}{w} = 0\bigr\}
\]
and
\[
 \Tilde{\mathcal{A}} (u) = \inf_{\Tilde{\mathcal{N}}_{\mathrm{nod}}} \Tilde{\mathcal{A}}.
\]

\end{claim}
\begin{proofclaim}
This follows from the fact that if \(u \in H^1_\mathrm{odd} (\Rset^N)\), then \(u \vert_{\Rset^N_+} \in H^1_0 (\Rset^N_+)\) and by direct computation of the integrals.
\end{proofclaim}

\begin{claim}
\label{claimPositive}
One has either \(u > 0\) almost everywhere on \(\Rset^N_+\) or \(u < 0\) almost everywhere on \(\Rset^N_+\).
\end{claim}
\begin{proofclaim}
Let \(w = u\vert_{\Rset^N_+}\). We observe that \(\abs{w} \in H^1_0 (\Rset^N_+)\),
\begin{align*}
 \Tilde{\mathcal{A}} (\abs{w}) & = \Tilde{\mathcal{A}} (w)  = c_{\mathrm{odd}} &
 & \text{and} &
 \dualprod{\Tilde{\mathcal{A}}' (\abs{w})}{\abs{w}}
 =  \dualprod{\Tilde{\mathcal{A}}' (w)}{w}.
\end{align*}
Therefore, if we define \(\Bar{u} \in H^1_{\mathrm{odd}} (\Rset^N)\) for almost every 
\(x = (x', x_N) \in \Rset^N\) by 
\[
 \Bar{u} (x)=
 \begin{cases}
   \abs{w} (x', x_N) & \text{if \(x_N > 0\)},\\
   -\abs{w} (x', x_N) & \text{if \(x_N < 0\)},
 \end{cases}
\]
the function \(\Bar{u}\) is a weak solution to the Choquard equation \eqref{eqChoquard}.
This function \(\Bar{u}\) is thus of class \(C^2\) \cite{MorVanSchaft13}*{proposition 4.1} (see also \cite{CingolaniClappSecchi2012}*{lemma A.10}) and, in the classical sense, it satisfies
\[
 -\Delta \Bar{u} + \Bar{u} \ge 0 \qquad \text{in \(\Rset^N_+\)}.
\]
By the usual strong maximum principle for classical supersolutions, we conclude that \(\abs{u} = \Bar{u} > 0\) in \(\Rset^N_+\).
Since the function \(u\) was also a solution to the Choquard equation \eqref{eqChoquard}, it is also a continuous function, and we have thus either \(u = \abs{u} > 0\) in \(\Rset^N_+\) or \(u = - \abs{u} < 0\) in \(\Rset^N_+\).
\end{proofclaim}

\begin{claim}
\label{claimSymmetry}
The solution \(u\) is axially symmetric with respect to an axis parallel to \(\{0\} \times \Rset \subset \Rset^N\).
\end{claim}
\begin{proofclaim}
Let \(H\) be a closed affine half-space perpendicular to \(\partial \Rset^{N}_+\) and let \(\sigma_H : \Rset^N \to \Rset^N\) be the reflection with respect to \(\partial H\). We recall that the polarization or two-point rearrangement with respect to \(H\) of \(w\) is the function \(w^H : \Rset^N \to \Rset\) defined for each \(x \in \Rset^N\) by \citelist{\cite{BrockSolynin2000}\cite{Baernstein1994}}
\[
 w^H (x) = 
 \begin{cases}
   \max \bigl(w (x), w (\sigma_H (x))\bigr) & \text{if \(x \in H\)},\\
   \min \bigl(w (x), w (\sigma_H (x))\bigr) & \text{if \(x \in \Rset^N \setminus H\)}.
 \end{cases}
\]
Since \(\partial H\) is perpendicular to \(\Rset^{N - 1} \times \{0\}\), we have \(\sigma_H (\Rset^N_+) = \Rset^N_+\) so that 
\(w^H \in H^1_0 (\Rset^N)\) and \cite{BrockSolynin2000}*{lemma 5.3} 
\[
 \int_{\Rset^N_+} \abs{\nabla w^H}^2 + \abs{w^H}^2
 = \int_{\Rset^N_+} \abs{\nabla w}^2 + \abs{w}^2.
\]
Moreover, we also have
\begin{multline}
\label{eqPolarizationNonlocal}
\int_{\Rset^N_+} \int_{\Rset^N_+} K (\abs{x' - y'}, x_N, y_N) \abs{w^H (x)}^p \abs{w^H (y)}^p \dif x \dif y\\
= \frac{1}{2} \int_{\Rset^N} \bigl(I_\alpha \ast (\abs{u}^p)^H\bigr) (\abs{u}^p)^H 
 \ge \frac{1}{2} \int_{\Rset^N_+} \bigl(I_\alpha \ast \abs{u}^p\bigr)\, \abs{u}^p\\
 =\int_{\Rset^N_+} \int_{\Rset^N_+} K (\abs{x' - y'}, x_N, y_N)\, \abs{w (x)}^p\, \abs{w (y)}^p \dif x \dif y,
\end{multline}
with equality if and only if either \(\abs{u}^H = \abs{u}\) almost everywhere on \(\Rset^N_+\) or 
\(\abs{u}^H = \abs{u} \circ \sigma_H\) almost everywhere on \(\Rset^N_+\) \cite{MorVanSchaft13}*{lemma 5.3} (see also \citelist{\cite{Baernstein1994}*{corollary 4}\cite{VanSchaftingenWillem2004}*{proposition 8}}), or equivalently,
since by claim~\ref{claimPositive} \(w\) has constant sign on \(\Rset^N\), either \(w^H = w\) almost everywhere on \(\Rset^N_+\) or \(w^H = w \circ \sigma_H\) almost everywhere on \(\Rset^N_+\).

If the inequality \eqref{eqPolarizationNonlocal} was strict, then, since \(p > 1\) there would exist \(\tau \in (0, 1)\) such that \(\tau w^H \in \Tilde{\mathcal{N}}\) and we would have 
\[
 \mathcal{A} (\tau u^H) < \Tilde{\mathcal{A}} (w) = c_{\mathrm{odd}},
\]
in contradiction with claim~\ref{claimRewrite}. 

We have thus proved that for every affine half-space \(H \subset \Rset^N\) whose boundary \(\partial H\) is perpendicular to \(\partial \Rset^{N}_+\), either 
\(w^H = w\) almost everywhere on \(\Rset^N_+\) or \(w^H = w \circ \sigma_H\) almost everywhere on \(\Rset^N_+\). This implies that \(w\) is axially symmetric with respect to an axis perpendicular to \(\partial \Rset^N_+\) \citelist{\cite{MorVanSchaft13}*{lemma 5.3}\cite{VanSchaftingenWillem2008}*{proposition 3.15}}, which is equivalent to the claim.
\end{proofclaim}

The proposition follows directly from claims~\ref{claimPositive} and \ref{claimSymmetry}.
\end{proof}

\section{Minimal action nodal solution}
This section is devoted to the proof of theorem~\ref{theoremNod} on the existence of a least action nodal solution.

\subsection{Minimax principle}
We begin by observing that the counterpart of proposition~\ref{propositionOddNondegenerate} holds.

\begin{proposition}[Nondegeneracy of the level]
\label{propositionNodNondegenerate}
If  \(\frac{N - 2}{N + \alpha} \le \frac{1}{p} \le \frac{N}{N + \alpha}\), then 
\[
  c_{\mathrm{nod}} > 0.
\]
\end{proposition}

\begin{proof}
In view of the inequality \(c_0 > 0\) \cite{MorVanSchaft13}, it suffices to 
note that since \(\mathcal{N}_{\mathrm{nod}} \subset \mathcal{N}_0\) we have 
\(c_{\mathrm{nod}} \ge c_0\). 
\end{proof}
We first reformulate the minimization problem as a minimax problem.

\begin{proposition}[Minimax principle]
\label{propositionSimpleVariational}
If \(\frac{N - 2}{N + \alpha} \le  \frac{1}{p} < 
\frac{1}{2}\), then for every \(\varepsilon > 0\), 
\[
  c_{\mathrm{nod}} = \inf_{\gamma \in \Gamma} \sup_{\Bset^2} \mathcal{A} 
\circ \gamma,
\]
where 
\begin{multline*}
  \Gamma = \Bigl\{ \gamma \in C \bigl(\Bset^2; H^1_0 (\Rset^N)\bigr) : \xi 
\bigl(\gamma(\partial \Bset^2)\bigr) \not \ni 0, \, \deg (\xi \circ 
\gamma ) = 1\\ \text{ and } (\mathcal{A} \circ 
\gamma)^\frac{p - 1}{p - 2} \le c_\mathrm{nod}^\frac{p - 1}{p - 2} + 
\varepsilon - c_0^\frac{p - 1}{p - 2} \text{on \(\partial \Bset^2\)}\Bigr\},
\end{multline*}
where the map \(\xi = (\xi_+, \xi_-) \in C \bigl(H^1 (\Rset^N); \Rset^2\bigr)\) 
is defined for each \(u \in H^1 (\Rset^N)\) by
\[
  \xi_{\pm} (u) = 
  \left\{
  \begin{aligned}&\frac{\displaystyle\int_{\Rset^N} (I_\alpha \ast 
\abs{u}^p) \abs{u_\pm}^p }{\displaystyle\int_{\Rset^N} \abs{\nabla u_\pm}^2 + 
\abs{u_\pm}^2} 
- 1 &  & \text{if \(u_\pm \ne 0\)},\\
& - 1 & & \text{if \(u_\pm = 0\)}.
\end{aligned}
\right.
\]
Moreover, for every \(\gamma \in \Gamma\), 
\(\mathcal{N}_{\mathrm{nod}} \cap \gamma (\Bset^2) \ne 0\).
\end{proposition}

In this statement \(\Bset^2\) denotes the closed unit disc in the plane 
\(\Rset^2\) and \(\deg\) is the classical topological degree of Brouwer, 
or equivalently, the winding number (see for example 
\citelist{\cite{Schechter2004}*{chapter 6}\cite{MawhinWillem1989}*{\S 5.3}}).

The continuity of the map \(\xi\) on the subset of constant-sign functions in 
\(H^1 (\Rset^N)\) follows from the Hardy--Littlewood--Sobolev inequality 
\eqref{eqHLS} and 
the classical Sobolev inequality, and requires the assumption \(p > 2\).

The map \(\xi\) is the nonlocal counterpart of a map appearing in the variational characterization of least action nodal solutions by Cerami, Solimini and Struwe for local 
Schr\"odinger type problems \cite{CeramiSoliminiStruwe1986}, which is done in the framework of critical point theory in ordered spaces whereas our minimax principle works in the more classical 
framework of Banach spaces.

\begin{proof}[Proof of proposition~\ref{propositionSimpleVariational}]
We denote the right-hand side in the equality to be proven as \(\Tilde{c}\) and 
we first prove that that 
\( \Tilde{c} \ge c_\mathrm{nod}\). 
Let \(\gamma \in \Gamma\). Since \(\deg (\xi \circ 
\gamma ) = 1\), by the existence property of the 
degree, there exists \(t^* \in \Bset^2\) 
such that \((\xi \circ \gamma) (t^*) = 0\). 
It follows then that \(\gamma (t_*) \in \mathcal{N}_{\mathrm{nod}} = \xi^{-1} 
(0)\) and thus 
\[
 \sup_{\Bset^2} \mathcal{A} \circ \gamma \ge \gamma (t_*) \ge c_{\mathrm{nod}},
\]
so that \(\Tilde{c} \ge c_{\mathrm{nod}}\).

We now prove that \(\Tilde{c} \le c_{\mathrm{nod}}\). 
For a given \(u \in \mathcal{N}_{\mathrm{nod}}\), we define the map 
\(\Tilde{\gamma} : [0, \infty)^2 \to H^1 (\Rset^N)\) for every \((t_+, t_-)\in 
[0, \infty)^2\)
by 
\[
 \Tilde{\gamma} (t_+, t_-) = t_+^\frac{1}{p} u^+ + 
t_-^\frac{1}{p} u^-.
\]
We compute for each \((t_+, t_-) \in [0, \infty)^2\)
\begin{multline}
\label{eqExplicittptm}
  \mathcal{A}\bigl( \Tilde{\gamma} (t_+, t_-)\bigr)
  =\frac{t_+^{2/p}}{2} \int_{\Rset^N} \abs{\nabla u^+}^2 +  \abs{ u^+}^2
+ \frac{t_-^{2/p}}{2} \int_{\Rset^N} \abs{\nabla u^-}^2+\abs{ u^-}^2\\
- \frac{1}{2 p} \int_{\Rset^N} \bigabs{I_{\alpha/2} \ast (t_+ \abs{u^+}^p + t_- 
\abs{u^-}^p)}^2.
\end{multline}
The function \(\mathcal{A} \circ \Tilde{\gamma}\) is thus strictly concave and 
\((\mathcal{A} \circ \Tilde{\gamma})' (1, 1) = 0\). Hence, \((1, 1)\) is the 
unique 
maximum point of the function \(\mathcal{A} \circ \Tilde{\gamma}\).

We also have in particular
\[
 \mathcal{A} \bigl(\Tilde{\gamma} (t_+, 0)\bigr)
 = \frac{t_+^{2/p}}{2} \int_{\Rset^N} \bigl(I_\alpha \ast \abs{u}^p\bigr) 
\abs{u^+}^p - \frac{t_+^2}{2 p}\int_{\Rset^N} \bigl(I_\alpha \ast 
\abs{u^+}^p\bigr) 
\abs{u^+}^p,
\]
and therefore
\begin{equation}
 \label{eqieu}
  \mathcal{A} \bigl(\Tilde{\gamma} (t_+, 0)\bigr) \le \Bigl(\frac{1}{2} - 
\frac{1}{2 p}\Bigr) 
\frac{\Bigl(\displaystyle\int_{\Rset^n} \bigl(I_\alpha \ast \abs{u}^p\bigr) 
\abs{u^+}^p\Bigr)^\frac{p}{p - 1}}{\Bigl(\displaystyle \int_{\Rset^n} 
\bigl(I_\alpha 
\ast 
\abs{u^+}^p\bigr) \abs{u^+}^p\Bigr)^\frac{1}{p - 1}}.
\end{equation}
By the semigroup property of the Riesz potential \(I_\alpha = I_{\alpha/2}\ast 
I_{\alpha/2}\) (see for example \cite{LiebLoss2001}*{theorem 5.9}) and by the 
Cauchy--Schwarz inequality,
\begin{multline}
 \label{eqiea}
  \int_{\Rset^N} \bigl(I_\alpha \ast \abs{u}^p\bigr)  \abs{u^+}^p = 
\int_{\Rset^N} 
(I_{\alpha/2} \ast \abs{u^+}^p)(I_{\alpha/2} \ast \abs{u}^p) \\
  \le \Bigl(\int_{\Rset^N} \bigabs{I_{\alpha/2} \ast 
\abs{u}^p}^2\Bigr)^\frac{1}{2} \Bigl(\int_{\Rset^N} \bigabs{I_{\alpha/2} \ast 
\abs{u^+}^p}^2\Bigr)^\frac{1}{2}\\
  = \Bigl(\int_{\Rset^N}\bigl(I_\alpha \ast \abs{u}^p\bigr)  \abs{u}^p 
)\Bigr)^\frac{1}{2}\Bigl(\int_{\Rset^N} (I_\alpha \ast 
\abs{u^+}^p)  \abs{u^+}^p )\Bigr)^\frac{1}{2},
\end{multline}
and therefore by \eqref{eqieu} and \eqref{eqiea}
\[
  \mathcal{A} \bigl(\Tilde{\gamma} (t_+, 0)\bigr) \le \Bigl(\frac{1}{2} - 
\frac{1}{2 p}\Bigr) 
\Bigl(\int_{\Rset^n} \bigl(I_\alpha \ast \abs{u}^p\bigr) 
\abs{u^+}^p\Bigr)^\frac{p - 2}{p - 1} 
\Bigl(\int_{\Rset^n} \bigl(I_\alpha \ast \abs{u}^p\bigr) 
\abs{u}^p\Bigr)^\frac{1}{p - 1}
\]
We deduce therefrom that for every \((t_+, t_-) \in [0, \infty)^2\),
\begin{equation}
\label{ineqSplit}
  \mathcal{A} \bigl(t_+^{1/p} u^+ \bigr)^\frac{p - 1}{p - 2} 
  + \mathcal{A} \bigl( t_-^{1/p} u^-\bigr)^\frac{p - 1}{p - 2}
  \le \mathcal{A} (u)^\frac{p - 1}{p - 2}.
\end{equation}
Since \(u_\pm \ne 0\), we have
\[
 \sup_{t_\pm \in [0, \infty)} \mathcal{A} \bigl(t_\pm^{1/p} u^\pm\bigr) \ge c_0,
\]
we conclude that 
\[
 \sup_{t \in \partial ([0, \infty)^2)} (\mathcal{A} \circ 
\Tilde{\gamma})^\frac{p - 1}{p - 2} 
 \le \mathcal{A} (u)^\frac{p - 1}{p - 2} - c_0^\frac{p - 1}{p - 2}.
\]
Moreover, we have by \eqref{eqExplicittptm}
\[
 \lim_{\abs{t} \to \infty} \mathcal{A} \bigl(\Tilde{\gamma}(t)\bigr) = - 
\infty.
\]

It remains to compute the degree of the map \(\xi \circ \Tilde{\gamma}\) on a 
suitable set homeomorphic to \(\Bset^2\). We compute for each \((t_+, t_-) \in 
[0, \infty)^2\),
 since \(u \in \mathcal{N}_\mathrm{nod}\)
\begin{multline*}
 \bigscalprod{(t_+, t_-)}{\xi (\Tilde{\gamma} (t_+, t_-))}\\
 = t_+^{3 - \frac{2}{p}} \frac{\displaystyle \int_{\Rset^N} (I_\alpha \ast 
\abs{u^+}^p) \abs{u^+}^p}{\displaystyle \int_{\Rset^N} (I_\alpha \ast 
\abs{u}^p) \abs{u^+}^p}
+
t_-^{3 - \frac{2}{p}} \frac{\displaystyle \int_{\Rset^N} (I_\alpha \ast 
\abs{u^-}^p) \abs{u^-}^p}{\displaystyle \int_{\Rset^N} (I_\alpha \ast 
\abs{u}^p) \abs{u^-}^p} - t_+ - t_-
\end{multline*}
Since \(p > 2\), we can now take \(R > \sqrt{2}\) large enough so that if \(t 
\in [0, \infty)^2 \cap \partial B_R\), then 
\begin{align*}
 \bigscalprod{t}{\xi (\Tilde{\gamma} (t))} & > 0 &
 &\text{ and } &
 \bigl(\mathcal{A} \circ \Tilde{\gamma} (t)\bigr)^\frac{p - 1}{p - 2}
  \le \mathcal{A} (u)^\frac{p - 1}{p - 2} 
  - c_0^\frac{p - 1}{p - 2}.
\end{align*}
We now define the homotopy \(H : [0, 1] \times [0, \infty)^2\) for each 
\((\tau, t_+, t_-) \in [0, 1] \times [0, \infty)^2\) by 
\[
  H (\tau, t)
  = \tau (\xi \circ \Tilde{\gamma})(t) + (1 - \tau) (t_+ - 1, t_- - 1).
\]
By the choice of \(R\), for every \((\tau, t) \in [0, 1] \times \partial ([0, 
\infty)^2 \cap B_R)\), \(H (\tau, t) \ne 0\), and thus by the homotopy 
invariance property of the degree, \(\deg (\xi \circ \Tilde{\gamma} 
\vert_{(0, \infty)^2 \cap B_R}) = 1\).
If we set \(\gamma = \Tilde{\gamma} \circ \Phi\), where \(\Phi : \Bset^2 
\to [0, \infty) \cap \Bar{B}_R\) is an orientation preserving homeomorphism,
we have \(\gamma \in \Gamma\) and \(\sup_{\Bset^2} \mathcal{A} \circ \gamma = 
\mathcal{A} (u)\).

We have thus proved that if \(u \in \mathcal{N}_\mathrm{nod}\) and if 
\(\mathcal{A} (u)^\frac{p - 1}{p - 2} \le c_\mathrm{nod}^\frac{p - 1}{p - 2} 
+ \varepsilon\), then 
\[
 \mathcal{A} (u) \ge \Tilde{c},
\]
from which we deduce that \(c_{\mathrm{nod}} \ge \Tilde{c}\).
\end{proof}

We would like to point out that the inequality \eqref{ineqSplit} in the proof 
of proposition~\ref{propositionSimpleVariational} gives a \emph{lower bound} on 
the level \(c_\mathrm{nod}\) that refines the degeneracy given for 
\(p > 2\) by theorem~\ref{theoremDegeneracy}.

\begin{corollary}
\label{corollaryLowerBound}
If \(\frac{N - 2}{N + \alpha} \le  \frac{1}{p} \le 
\frac{1}{2}\), then 
\[
  c_{\mathrm{nod}} \ge 2^{\frac{p - 2}{p - 1}} c_0.
\]
\end{corollary}

In particular, corollary \eqref{corollaryLowerBound} allows
theorem~\ref{theoremDegeneracy} to hold when \(p = 2\).

\begin{proposition}[Existence of a Palais-Smale sequence]
\label{propositionNodalPalaisSmaleExistence}
If \(\frac{N - 2}{N + \alpha} \le \frac{1}{p} < \frac{1}{2}\), then 
there exists a sequence \((u_n)_{n \in \N}\) in \(H^1 (\Rset^N)\) 
such that 
\begin{align*}
  \mathcal{A} (u_n) & \to c_{\mathrm{nod}}, &
  & \dist (u_n, \mathcal{N}_{\mathrm{nod}}) \to 0 & 
  &\text{ and} &
  \mathcal{A}' (u_n) & \to 0 \quad \text{in \(H^1 (\Rset^N)'\)},
\end{align*}
as \(n \to \infty\).
\end{proposition}

\begin{proof}
We take \(\Gamma\) given by proposition~\ref{propositionSimpleVariational} with 
\(\varepsilon < c_0^\frac{p - 1}{p - 2}\).
The location theorem \cite{Willem1996}*{theorem 2.20} (see also 
\cite{BrezisNirenberg1991}*{theorem 2}\cite{Schechter2009}*{theorem 2.12}) is 
applicable and gives the conclusion.
\end{proof}

\subsection{Convergence of the Palais--Smale sequence}
We prove that Palais--Smale sequences at the level \(c_\mathrm{nod}\) and 
localized near the Nehari nodal set \(\mathcal{N}_{\mathrm{nod}}\) converge 
strongly up to a subsequence and up to translations.

\begin{proposition}
\label{propositionNodalPalaisSmaleCondition}
Let \((u_n)_{n \in \N}\) be a sequence in \(H^1_{\mathrm{odd}} (\Rset^N)\) such 
that, as \(n \to \infty\),
\begin{align*}
  \mathcal{A} (u_n) & \to c_{\mathrm{nod}}, &
  & \dist (u_n, \mathcal{N}_{\mathrm{nod}}) \to 0 &
  &\text{ and} &
  \mathcal{A}' (u_n) & \to 0 \quad \text{in \(H^1 (\Rset^N)'\)}.
\end{align*}
If \(\frac{N - 2}{N + \alpha} < \frac{1}{p} < \frac{1}{2}\) and if 
\[
  c_{\mathrm{nod}} < 2 c_0,
\]
then there exists a sequence of points \((a_n)_{n \in \N}\) in \(\Rset^N\) such 
that the subsequence
\(\bigl(u_{n_k} (\cdot - a_{n_k})\bigr)_{k \in \N}\) converges strongly in 
\(H^1 
(\Rset^N)\) to \(u \in H^1 (\Rset^N)\).
Moreover
\begin{align*}
  \mathcal{A} (u) &= c_{\mathrm{nod}}, &
  u &\in \mathcal{N}_{\mathrm{nod}}&
  &\text{ and }&
  \mathcal{A}'(u) &= 0 \quad \text{in \(H^1 (\Rset^N)'\)}.
\end{align*}
\end{proposition}

Palais--Smale conditions have been already proved by concentration--com\-pact\-ness arguments for local semilinear elliptic problems \citelist{\cite{FurtadoMaiaMedeiros2008}\cite{CeramiSoliminiStruwe1986}}.

\begin{proof}[Proof of proposition~\ref{propositionNodalPalaisSmaleCondition}]
We shall proceed through several claims on the sequence \((u_n)_{n \in \N}\).
\setcounter{claim}{0}

\begin{claim}
The sequence \((u_n)_{n \in \N}\) is bounded in the space \(H^1 (\Rset^N)\).
\end{claim}
\begin{proofclaim}
We write, as \(n \to \infty\),
\[
\begin{split}
  \Bigl(\frac{1}{2} - \frac{1}{2 p}\Bigr)\int_{\Rset^N} \abs{\nabla u_n}^2 + 
\abs{u_n}^2 &= \mathcal{A} (u_n) - \frac{1}{2 p} \dualprod{\mathcal{A}' 
(u_n)}{u_n}\\
  &= \mathcal{A} (u_n) + o \biggl( \Bigl(\int_{\Rset^N} \abs{\nabla u_n}^2 + 
\abs{u_n}^2 \Bigr)^\frac{1}{2}\biggr).
\end{split}
\]
from which the claim follows.
\end{proofclaim}

We now show that neither positive nor the negative parts of the sequence \((u_n)_{n \in 
\N}\) tend to \(0\).

\begin{claim}
\label{claimUniformContinuity}%
The functional \(u \in H^1 (\Rset^N) \mapsto \int_{\Rset^N} 
(I_\alpha \ast \abs{u}^p) \abs{u^\pm}^p\) is uniformly continuous on bounded 
subsets of the space \(H^1 (\Rset^N)\).
\end{claim}

We bring to the attention of the reader that the related 
map \(u \in H^1 (\Rset^N) \mapsto \int_{\Rset^N} \abs{\nabla u^\pm}^2 + 
\abs{u^\pm}^2\) \emph{is not uniformly continuous} on bounded sets.

\begin{proofclaim}
For every \(u, v \in H^1 (\Rset^N)\), we have 
\begin{multline*}
 \int_{\Rset^N} 
(I_\alpha \ast \abs{u}^p) \abs{u^\pm}^p
- \int_{\Rset^N} (I_\alpha \ast \abs{v}^p) \abs{v^\pm}^p\\
= \frac{1}{2} \int_{\Rset^N} \bigl(I_\alpha \ast (\abs{u}^p + \abs{v}^p)\bigr)\bigl(\abs{u^\pm}^p - \abs{v^\pm}^p\bigr)\\
+ \frac{1}{2}\int_{\Rset^N} \bigl(I_\alpha \ast (\abs{u}^p - \abs{v}^p)\bigr)\bigl(\abs{u^\pm}^p + \abs{v^\pm}^p\bigr).
\end{multline*}
By the classical Hardy--Littlewood--Sobolev inequality \eqref{eqHLS}, we obtain 
\begin{multline*}
 \Bigabs{\int_{\Rset^N} 
(I_\alpha \ast \abs{u}^p) \abs{u^\pm}^p
- \int_{\Rset^N} (I_\alpha \ast \abs{v}^p) \abs{v^\pm}^p}\\
\le C \Bigl(\int_{\Rset^N} \bigl(\abs{u}^p + \abs{v}^p\bigr)^\frac{2 N}{N + \alpha}\Bigr)^{\frac{N + \alpha}{2N}}
\Bigl(\int_{\Rset^N} \bigabs{\abs{u^\pm}^p - \abs{v^\pm}^p}^\frac{2 N}{N + \alpha}  \Bigr)^\frac{N + \alpha}{2 N}\\
+ C \Bigl(\int_{\Rset^N} \bigl(\abs{u^\pm}^p + \abs{v^\pm}^p\bigr)^\frac{2 N}{N + \alpha}\Bigr)^{\frac{N + \alpha}{2N}}
\Bigl(\int_{\Rset^N} \bigabs{\abs{u}^p - \abs{v}^p}^\frac{2 N}{N + \alpha}  \Bigr)^\frac{N + \alpha}{2 N}.
\end{multline*}
Since for every \(s, t \in \Rset\), one has \(\bigabs{\abs{s^\pm}^p - \abs{t^\pm}^p}
\le \abs{\abs{s}^p - \abs{t}^p}\) and \(\abs{s^\pm}^p + \abs{t^\pm}^p \le \abs{s}^p + \abs{t}^p\), the latter estimate can be simplified to
\begin{multline*}
 \Bigabs{\int_{\Rset^N} 
(I_\alpha \ast \abs{u}^p) \abs{u^\pm}^p
- \int_{\Rset^N} (I_\alpha \ast \abs{v}^p) \abs{v^\pm}^p}\\
\le 2C \Bigl(\int_{\Rset^N} \bigl(\abs{u}^p + \abs{v}^p\bigr)^\frac{2 N}{N + \alpha}\Bigr)^{\frac{N + \alpha}{2N}}
\Bigl(\int_{\Rset^N} \bigabs{\abs{u}^p - \abs{v}^p}^\frac{2 N}{N + \alpha}  \Bigr)^\frac{N + \alpha}{2 N}.
\end{multline*}
Next, for each \(s, t \in \Rset\), we have, since \(p \ge 2\),
\[
\begin{split}
  \bigabs{\abs{s}^p - \abs{t}^p} & \le p \abs{s - t} \int_0^1 p \abs{\tau s + (1 - \tau)t}^{p - 1} \dif \tau\\
  & \le p \abs{s - t} \int_0^1 \tau \abs{s}^{p - 1} + (1 - \tau) \abs{t}^{p - 1}\dif \tau \\
  & = \frac{p}{2} \abs{s - t} \bigl(\abs{s}^{p - 1} + \abs{t}^{p - 1} \bigr).
\end{split}
\]
Hence, we have,
\begin{multline*}
 \Bigabs{\int_{\Rset^N} 
(I_\alpha \ast \abs{u}^p) \abs{u^\pm}^p
- \int_{\Rset^N} (I_\alpha \ast \abs{v}^p) \abs{v^\pm}^p}\\
\le 2 C \Bigl(\frac{p}{2} \Bigr)^\frac{N + \alpha}{2 N}
\Bigl(\int_{\Rset^N} \bigl(\abs{u}^p + \abs{v}^p\bigr)^\frac{2 N}{N + \alpha}\Bigr)^{\frac{N + \alpha}{2N}}\\
\times 
\Bigl(\int_{\Rset^N} \bigl(\abs{u}^{p - 1} + \abs{v}^{p - 1}\bigr)^\frac{2N}{N + \alpha} \abs{u - v}^\frac{2 N}{N + \alpha}  \Bigr)^\frac{N + \alpha}{2 N}.
\end{multline*}
Therefore, by the H\"older inequality,
\begin{multline*}
 \Bigabs{\int_{\Rset^N} 
(I_\alpha \ast \abs{u}^p) \abs{u^\pm}^p
- \int_{\Rset^N} (I_\alpha \ast \abs{v}^p) \abs{v^\pm}^p}\\
\le C'
\Bigl(\int_{\Rset^N} \abs{u}^\frac{2 N p}{N + \alpha} + \abs{v}^\frac{2 N p}{N + \alpha}\Bigr)^{\frac{N + \alpha}{N}(1 - \frac{1}{2 p})}
\Bigl(\int_{\Rset^N} \abs{u - v}^\frac{2 Np}{N + \alpha} \Bigr)^\frac{N + \alpha}{2 N p}.
\end{multline*}
This shows that the map is uniformly continuous on \(L^\frac{2 N p}{N + \alpha} (\Rset^N)\).
Since by assumption, \(\frac{N - 2}{N + \alpha} \le \frac{1}{p} \le \frac{N}{N + \alpha}\), in view of the classical Sobolev embedding theorem, the embedding \(H^1 (\Rset^N) \subset L^\frac{2 N p}{N + \alpha} (\Rset^N)\) is uniformly continuous, and the claim follows.
\end{proofclaim}

\begin{claim}
\label{claimLowerBound}
We have 
\[
 \liminf_{n \to \infty} \int_{\Rset^N} \abs{\nabla u_n^\pm}^2 + \abs{u_n^\pm}^2
 = \liminf_{n \to \infty} \int_{\Rset^N} \bigl(I_\alpha \ast \abs{u_n}^p\bigr) 
\abs{u_n^\pm}^p > 0.
\]
\end{claim}
\begin{proofclaim}
First we observe that if \(v \in \mathcal{N}_{\mathrm{nod}}\), then by the 
Hardy--Littewood--Sobolev inequality \eqref{eqHLS}, the Sobolev inequality and the definition 
of the nodal Nehari set \(\mathcal{N}_{\mathrm{nod}}\), we have
\[
\begin{split}
 \int_{\Rset^N} \bigl(I_\alpha \ast \abs{v}^p\bigr) 
\abs{v^\pm}^p
&\le C \Bigl(\int_{\Rset^N} \abs{v}^\frac{2 N p}{N + \alpha} \Bigr)^\frac{N + 
\alpha}{2 N}\Bigl(\int_{\Rset^N} \abs{v^\pm}^\frac{2 N p}{N + \alpha} 
\Bigr)^\frac{N + 
\alpha}{2 N}\\
&\le C' \Bigl(\int_{\Rset^N} \abs{\nabla v}^2 + \abs{v}^2 
\Bigr)^\frac{p}{2}
\Bigl(\int_{\Rset^N} \abs{\nabla v^\pm}^2 + \abs{v^\pm}^2 
\Bigr)^\frac{p}{2}\\
&= C' \Bigl(\int_{\Rset^N} \abs{\nabla v}^2 + \abs{v}^2 
\Bigr)^\frac{p}{2}
\Bigl(\int_{\Rset^N} \bigl(I_\alpha \ast \abs{v}^p\bigr) 
\abs{v^\pm}^p\Bigr)^\frac{p}{2}.
\end{split}
\]
Since \(v^\pm \ne 0\), we deduce that 
\[
 \inf_{v \in \mathcal{N}_\mathrm{nod}} \Bigl(\int_{\Rset^N} \abs{\nabla v}^2 + 
\abs{v}^2 
\Bigr)^\frac{p}{2}
\Bigl(\int_{\Rset^N} \bigl(I_\alpha \ast \abs{v}^p\bigr) 
\abs{v^\pm}^p\Bigr)^\frac{p-2}{2} > 0.
\]

Since the sequence \((u_n)_{n \in \N}\) is bounded in the space \(H^1 
(\Rset^N)\) and since \(\lim_{n \to \infty} \dist (u_n, 
\mathcal{N}_\mathrm{odd}) = 0\), we deduce from the lower bound above and from 
the uniform continuity property of claim~\ref{claimUniformContinuity} that 
\[
 \liminf_{n \to \infty} \int_{\Rset^N} \bigl(I_\alpha \ast \abs{u_n}^p\bigr) 
\abs{u_n^\pm}^p > 0.
\]
Since \(\lim_{n \to \infty} \dualprod{\mathcal{A}' (u_n)}{u_n^\pm} = 0\), the 
conclusion follows.
\end{proofclaim}

\begin{claim}
There exists \(R > 0\) such that 
\[
 \limsup_{n \to \infty} \sup_{a \in \Rset^N} \int_{B_R (a)} \abs{u_n^+}^\frac{2 
N p}{N + \alpha} \int_{B_R (a)} \abs{u_n^-}^\frac{2 N 
p}{N + \alpha} > 0.
\]
\end{claim}

\begin{proofclaim}
We assume by contradiction that for every \(R > 0\), 
\[
 \lim_{n \to \infty} \sup_{a \in \Rset^N} \int_{B_R (a)} \abs{u_n^+}^\frac{2 
N p}{N + \alpha} \int_{B_R (a)} \abs{u_n^-}^\frac{2 N 
p}{N + \alpha} = 0.
\]
Then by lemma~\ref{lemmaNewEstimate} below, since the sequences \((u_n^+)_{n 
\in 
\N}\) and \((u_n^-)_{n \in \N}\) are bounded in the space \(H^1 (\Rset^N)\), we have 
\[
 \lim_{n \to \infty} \int_{\Rset^N} \bigl(I_\alpha \ast \abs{u_n^+}^p\bigr) 
\abs{u_n^-}^p = 0.
\]

We now take \(t_{n, \pm} \in (0, \infty)\) such that \(t_{n, \pm} u_n^\pm \in 
\mathcal{N}_0\).
Since 
\[
\begin{split}
\displaystyle \int_{\Rset^N} 
\abs{\nabla u_n^\pm}^2 + \abs{u_n^\pm}^2
&= \int_{\Rset^N} \bigl(I_\alpha 
\ast \abs{u_n}^p\bigr) \abs{u_n^\pm}^p + \dualprod{\mathcal{A}' (u_n)}{u_n^\pm}\\
&= \int_{\Rset^N} \bigl(I_\alpha 
\ast \abs{u_n}^p\bigr) \abs{u_n^\pm}^p + o (1)
\end{split}
\]
and 
\[
\begin{split}
\int_{\Rset^N} \bigl(I_\alpha 
\ast \abs{u_n^\pm}^p\bigr) \abs{u_n^\pm}^p
 &= \int_{\Rset^N} \bigl(I_\alpha 
\ast \abs{u_n}^p\bigr) \abs{u_n^\pm}^p - \int_{\Rset^N} \bigl(I_\alpha 
\ast \abs{u_n^+}^p\bigr) \abs{u_n^-}^p\\
&= \int_{\Rset^N} \bigl(I_\alpha 
\ast \abs{u_n}^p\bigr) \abs{u_n^\pm}^p + o (1),
\end{split}
\]
it follows, in view of claim~\ref{claimLowerBound}, that \(\lim_{n \to \infty} 
t_{n, \pm} = 1\).
We compute
\begin{multline*}
 \mathcal{A} (t_{n, +} u_n^+ + t_{n, -} u_n^-)\\
 = \mathcal{A} (t_{n, +} u_n^+) + \mathcal{A} (t_{n, -} u_n^-)
 -\frac{t_{n,+}^p t_{n, -}^p}{p} \int_{\Rset^N} \bigl(I_\alpha \ast 
\abs{u_n^+}^p\bigr) \abs{u_n^-}^p.
\end{multline*}
and we deduce that 
\[
\begin{split}
 c_\mathrm{nod} &= \lim_{n \to \infty} \mathcal{A} (t_{n, +} u_n^+ + t_{n, -} 
u_n^-)\\
&\ge \liminf_{n \to \infty} \mathcal{A} (t_{n, +} u_n^+) + \liminf_{n 
\to \infty} \mathcal{A} (t_{n, -} u_n^-) \ge 2 c_0,
\end{split}
\]
in contradiction with the assumption of the proposition.
\end{proofclaim}

We now conclude the proof of the proposition. Up to a translation and a subsequence, we can assume 
that 
\[
 \liminf_{n \to \infty} \int_{B_R (a)} \abs{u_n^\pm}^\frac{2 
N p}{N + \alpha} \ge 0
\]
and that the sequence \((u_n)_{n \in \N}\) converges weakly to some \(u \in H^1 
(\Rset^N)\).
As in the proof of proposition~\ref{propositionOddPalaisSmaleCondition}, by the 
weak convergence and by the classical Rellich--Kondrashov compactness 
theorem, \(\mathcal{A}' (u) = 0\) and \(u^\pm \ne 0\), whence \(u \in 
\mathcal{N}_\mathrm{nod}\).
We also have by lower semicontinuity of the Sobolev norm under weak convergence
\[
\begin{split}
\lim_{n \to \infty} \mathcal{A} (u_n) &= \lim_{n \to \infty} \mathcal{A} (u_n) 
- \frac{1}{2 p} \dualprod{\mathcal{A}' (u_n)}{u_n}\\
& =\lim_{n \to \infty} \Bigl(\frac{1}{2} - \frac{1}{2 p} \Bigr)\int_{\Rset^N} 
\abs{\nabla u_n}^2 + \abs{u_n}^2\\
&\ge \Bigl(\frac{1}{2} - \frac{1}{2 p} \Bigr)\int_{\Rset^N} \abs{\nabla u}^2 + 
\abs{u}^2\\
&=\mathcal{A} (u) - \frac{1}{2 p} \dualprod{\mathcal{A}' (u)}{u} = \mathcal{A}(u),
\end{split}
\]
from which we deduce that \(\mathcal{A} (u) = c_\mathrm{nod}\) and the strong 
convergence of the sequence \((u_n)_{n \in \N}\) in the space \(H^1 (\Rset^N)\).
\end{proof}

\begin{lemma}
\label{lemmaNewEstimate}
If \(\frac{N - 2}{N + \alpha} \le \frac{1}{p} < \frac{N}{N + \alpha}\), then 
for every \(\beta \in \bigl(\alpha, \min (1, p - 1) N\bigr)\), there exists \(C 
> 0\) such 
that for every \(u, v \in H^1 (\Rset^N)\), 
\begin{multline*}
 \int_{\Rset^N} \bigl(I_\alpha \abs{u}^p\bigr) \abs{v}^p
 \le C \Bigl(\int_{\Rset^N} \abs{\nabla 
u}^2 + \abs{u}^2 \int_{\Rset^N}  \abs{\nabla 
v}^2 + \abs{v}^2 \Bigr)^\frac{1}{2}\\
\shoveright{\Bigl(\sup_{a \in \Rset^N} \int_{B_R 
(a)} \abs{u}^\frac{2 N p}{N + \alpha}\int_{B_R (a)} 
\abs{v}^\frac{2 N p}{N + \alpha}\Bigr)^{\frac{N + \alpha}{2 
N}(1 - \frac{1}{p})}}\\
+ \frac{C}{R^{\beta - \alpha}} \Bigl(\int_{\Rset^N} \abs{\nabla 
u}^2 + \abs{u}^2 \int_{\Rset^N}  \abs{\nabla 
v}^2 + \abs{v}^2 \Bigr)^\frac{p}{2}
\end{multline*}
 
\end{lemma}

When \(p = \frac{N + \alpha}{N}\), then \((p - 1)N = \alpha\) and there is no \(\beta\) that would satisfy the assumptions.

\begin{proof}[Proof of lemma~\ref{lemmaNewEstimate}]
We first decompose the integral as 
\[
 \int_{\Rset^N} \bigl(I_\alpha \abs{u}^p\bigr) \abs{v}^p
 =\int_{\Rset^N} \bigl((\chi_{B_{R/2}} I_\alpha) \ast \abs{u}^p\bigr) \abs{v}^p
 +\int_{\Rset^N} \bigl((\chi_{\Rset^N \setminus B_{R/2}} I_\alpha) \ast
\abs{u}^p\bigr) \abs{v}^p.
\]
We then observe that if \(\beta \in (\alpha, N)\), then 
\[
 \int_{\Rset^N} \bigl((\chi_{\Rset^N \setminus B_{R/2}} \ast I_\alpha) 
\abs{u}^p\bigr) 
\abs{v}^p \le \frac{C}{R^{\beta - \alpha}} \int_{\Rset^N} (I_\beta \ast 
\abs{u}^p) \abs{v}^p.
\]
If \(\beta < (p - 1) N\), then by the Hardy--Littewood--Sobolev inequality and 
by the Sobolev inequality, we have 
\[
 \int_{\Rset^N} \bigl((\chi_{\Rset^N \setminus B_{R/2}} I_\alpha) \ast 
\abs{u}^p\bigr) 
\abs{v}^p \le \frac{C'}{R^{\beta - \alpha}} \Bigl(\int_{\Rset^N} \abs{\nabla 
u}^2 + \abs{u}^2\Bigr)^\frac{p}{2}\Bigl(\int_{\Rset^N} \abs{\nabla v}^2 
+ \abs{v}^2\Bigr)^\frac{p}{2}.
\]
Next, we have 
\begin{multline*}
 \int_{\Rset^N} \bigl((\chi_{B_{R/2}} I_\alpha) \ast \abs{u}^p\bigr) \abs{v}^p\\
 \le \frac{C'}{R^N} \int_{\Rset^N} \int_{B_R (a)} \int_{B_R (a)} I_\alpha (x 
- y) \abs{u (x)}^p \abs{v (y)}^p \dif x \dif y \dif a.
\end{multline*}
For every \(a \in \Rset^N\), we have, by the Hardy--Littewood--Sobolev 
inequality \eqref{eqHLS} and the classical Sobolev inequality on the ball \(B_R 
(a)\),
\begin{multline*}
\int_{B_R (a)} \int_{B_R (a)} I_\alpha (x 
- y) \abs{u (x)}^p \abs{v (y)}^p \dif x \dif y\\
\shoveleft{\le C'' \Bigl(\int_{B_R (a)} \abs{u}^\frac{2 N p}{N + \alpha} 
\int_{B_R (a)} \abs{v}^\frac{2 N p}{N + \alpha}\Bigr)^\frac{N + \alpha}{2 N}}\\
\le C''' \Bigl(\int_{B_R (a)} \abs{u}^\frac{2 N p}{N + \alpha}\int_{B_R (a)} 
\abs{v}^\frac{2 N p}{N + \alpha}\Bigr)^{\frac{N + \alpha}{2 N}(1-\frac{1}{p})}\\
\times \Bigl(\int_{B_R (a)} \abs{\nabla u}^2 + \abs{u}^2 \int_{B_R (a)} \abs{\nabla 
v}^2 + \abs{v}^2 \Bigr)^\frac{1}{2}.
\end{multline*}
We now integrate this estimate with respect to \(a \in \Rset^N\) and apply the 
Cauchy--Schwarz inequality to 
obtain
\begin{multline*}
 \int_{\Rset^N} \bigl((\chi_{B_{R/2}} I_\alpha) 
\abs{u}^p\bigr) 
\abs{v}^p\\
\le \frac{C'''}{R^N} \Bigl(\int_{\Rset^N} \Bigl(\int_{B_R (a)}  \abs{\nabla 
u}^2 + \abs{u}^2\Bigr) \dif a\Bigr)^\frac{1}{2} \Bigl(\int_{\Rset^N} \Bigl(\int_{B_R (a)}  \abs{\nabla 
v}^2 + \abs{v}^2\Bigr) \dif a\Bigr)^\frac{1}{2}\\
\times \shoveright{\Bigl(\sup_{a \in \Rset^N} \int_{B_R 
(a)} \abs{u}^\frac{2 N p}{N + \alpha}\int_{B_R (a)} 
\abs{v}^\frac{2 N p}{N + \alpha}\Bigr)^{\frac{N + \alpha}{2 
N}(1-\frac{1}{p})}}\\
= C''''\Bigl(\int_{\Rset^N}  \abs{\nabla 
u}^2 + \abs{u}^2 \int_{\Rset^N} \abs{\nabla 
v}^2 + \abs{v}^2 \Bigr)^\frac{1}{2}\\
\times \Bigl(\sup_{a \in \Rset^N} \int_{B_R 
(a)} \abs{u}^\frac{2 N p}{N + \alpha}\int_{B_R (a)} 
\abs{v}^\frac{2 N p}{N + \alpha}\Bigr)^{\frac{N + \alpha}{2 N}(1-\frac{1}{p})}.
\end{multline*}
\end{proof}

\subsection{Existence of a minimal action nodal solution}
In order to prove theorem~\ref{theoremNod}, we finally establish the strict 
inequality.

\begin{proposition}
\label{propositionLevelsOddNod}
If \(\frac{N - 2}{N + \alpha} \le \frac{1}{p} \le \frac{N}{N + \alpha}\), then 
\[
 \mathcal{N}_\mathrm{odd} \subset \mathcal{N}_{\mathrm{nod}}.
\]
In particular,
\[
 c_\mathrm{nod} \le c_\mathrm{odd}.
\]
\end{proposition}
\begin{proof}
If \(u \in \mathcal{N}_\mathrm{odd}\),
then since \(u \in \mathcal{N}_0\),
\[
 \dualprod{\mathcal{A}' (u)}{u^+} + \dualprod{\mathcal{A}' (u)}{u^-} = 
\dualprod{\mathcal{A}' (u)}{u} = 0.
\]
On the other hand, since \(u \in H^1_\mathrm{odd} (\Rset^N)\), by the 
invariance of \(u\) under odd reflections,
\[
 \dualprod{\mathcal{A}' (u)}{u^+} = \dualprod{\mathcal{A}' (u)}{u^-},
\]
and the conclusion follows.
\end{proof}

We can now prove theorem~\ref{theoremNod} about the existence of minimal action 
nodal solutions.

\begin{proof}[Proof of theorem~\ref{theoremNod}]
Proposition~\ref{propositionNodalPalaisSmaleExistence} gives the existence of a 
localized Palais--Smale sequence \((u_n)_{n \in \N}\).
By propositions~\ref{propositionStrictInequality} and \ref{propositionLevelsOddNod},
the strict inequality \(c_\mathrm{nod} < 2 c_0\) holds. Hence we can apply 
proposition~\ref{propositionNodalPalaisSmaleCondition} to reach the conclusion.
The solution \(u\) is twice continuously differentiable by the Choquard equation's regularity theory \cite{MorVanSchaft13}*{proposition 4.1} (see also \cite{CingolaniClappSecchi2012}*{lemma A.10}).
\end{proof}

\subsection{Degeneracy in the locally sublinear case}

We conclude this paper by proving that \(c_{\mathrm{nod}} = c_0\) if \(p < 2\).

\begin{proof}[Proof of theorem~\ref{theoremDegeneracy}]
We observe that if \(u \in \mathcal{N}_0\), then \(\abs{u} \in \mathcal{N}_0\).
Together with a density argument, this shows that 
\[
 c_0 = \inf \bigl\{\mathcal{A} (u)\st u \in \mathcal{N}_0 \cap C^1_c (\Rset^N) 
\text{ and }  u \ge 0 \text{ on } \Rset^N\bigr\}.
\]
Let now \(u \in \mathcal{N}_0 \cap C^1_c (\Rset^N)\) such that \(u \ge 0\) on \(\Rset^N\). We choose a point \(a \not \in \supp u\) and a function \(\varphi \in C^1_c (\Rset^N) \setminus \{0\}\) such that \(\varphi \ge 0\) and we define for each \(\delta > 0\) the function \(u_\delta : \Rset^N \to \Rset\) for every \(x \in \Rset^N\) by
\[
  u_\delta(x) = u (x) - \delta^\frac{2}{2 - p} \varphi \Bigl(\frac{x - a}{\delta} \Bigr).
\]
We observe that, if \(\delta > 0\) is small enough, \(u_\delta^+ = u^+\).
% For such \(\delta\), we compute
% \begin{multline*}
%  \dualprod{\mathcal{A} (t_+ u_\delta^+ + t_- u_\delta^-)}{t_+ u_\delta^+}
%  = t_+^2 \int_{\Rset^N} \abs{\nabla u}^2 + \abs{u}^2
%  - t_+^{2 p} \int_{\Rset^N} \bigl(I_\alpha \ast \abs{u}^p\bigr) \abs{u}^p\\
%  - t_-^p t_+^p \delta^{N + \frac{2p}{2 - p}} \int_{\Rset^N} \bigl(I_\alpha \ast \abs{u}^p\bigr) (z + a \delta) \bigl(\varphi (z)\bigr)^p\dif z
% \end{multline*}
% and 
% \begin{multline*}
%  \dualprod{\mathcal{A} (t_+ u_\delta^+ + t_- u_\delta^-)}{t_- u_\delta^-}
%  = t_-^2 \delta^{N + \frac{2 p}{2 - p}} \int_{\Rset^N} \abs{\nabla \varphi}^2 + \delta^2 \abs{\varphi}^2
%  - t_-^{2 p} \delta^{N + \alpha + \frac{4 p}{2 - p}} \int_{\Rset^N} \bigl(I_\alpha \ast \abs{\varphi}^p\bigr) \abs{\varphi}^p\\
%  - t_-^p t_+^p \delta^{N + \frac{2p}{2 - p}} \int_{\Rset^N} \bigl(I_\alpha \ast \abs{u}^p\bigr) (z + a \delta) \bigl(\varphi (z)\bigr)^p\dif z.
% \end{multline*}
By a direct computation, we have \(t_+ u_\delta^+ + t_- u_\delta^- \in \mathcal{N}_\mathrm{nod}\) if and only if 
\begin{equation}
\label{eqDegeneracySystem}
\left\{
\begin{aligned}
 (t_+^{2 - p} - t_+^{p} )\int_{\Rset^N} \bigl(I_\alpha \ast \abs{u}^p\bigr) \abs{u}^p
 &= t_-^p \delta^{N + \frac{2p}{2 - p}} J_\delta,\\
 t_-^{2 - p} \int_{\Rset^N} \abs{\nabla \varphi}^2 + \delta^2 \abs{\varphi}^2
  &=t_+^p J_\delta + t_-^{p} \delta^{\alpha + \frac{2 p}{2 - p}} \int_{\Rset^N} \bigl(I_\alpha \ast \abs{\varphi}^p\bigr) \abs{\varphi}^p,
\end{aligned}
\right.
\end{equation}
where
\[
J_\delta = \int_{\Rset^N} \bigl(I_\alpha \ast \abs{u}^p\bigr) (\delta z + a ) \bigl(\varphi (z)\bigr)^p\dif z.
\]
We observe that when \(\delta = 0\), the system reduces to 
\[
\left\{
\begin{aligned}
 (t_+^{2 - p} - t_+^{p} )\int_{\Rset^N} \bigl(I_\alpha \ast \abs{u}^p\bigr) \abs{u}^p
 &= 0,\\
 t_-^{2 - p} \int_{\Rset^N} \abs{\nabla \varphi}^2
  &=t_+^p \bigl(I_\alpha \ast \abs{u}^p\bigr) (a) \int_{\Rset^N} \abs{\varphi}^p,
\end{aligned}
\right.
\]
which has a unique solution. By the implicit function theorem, for \(\delta > 0\) small enough there exists \((t_{+, \delta}, t_{-, \delta}) \in (0, \infty)^2\) that solves the system \eqref{eqDegeneracySystem} and such that 
\[
 \lim_{\delta \to 0}(t_{+, \delta}, t_{-, \delta})
 =\biggl(1,  \Bigl( \bigl(I_\alpha \ast \abs{u}^p\bigr) (a) \int_{\Rset^N} 
\abs{\varphi}^p \Big/ \int_{\Rset^N} \abs{\nabla \varphi}^2\Bigr)^\frac{1}{2 - 
p}\biggr).
\]
Since \((N - 2)(2 - p) > -4\), we have \(t_+ u_\delta^+ + 
t_- u_\delta^- \to u\) strongly in \(H^1 (\Rset^N)\) as \(\delta \to 0\), and it follows 
that 
\[
 \inf_{\mathcal{N}_\mathrm{nod}} \mathcal{A} \le \mathcal{A} (u),
\]
and thus \(c_{\mathrm{nod}} \le c_0\).

We assume now that the function \(u \in \mathcal{N}_\mathrm{nod}\) minimizes 
the action functional \(\mathcal{A}\) on the nodal Nehari set 
\(\mathcal{N}_\mathrm{nod}\).
Since \(c_0 = c_{\mathrm{nod}}\), we deduce that \(u\) also minimizes \(\mathcal{A}\)
over the Nehari manifold \(\mathcal{N}_0\).
By the properties of such groundstates \cite{MorVanSchaft13}*{proposition 5.1}, either \(u^+ = 0\) or \(u^- = 0\), in contradiction with the assumption \(u \in \mathcal{N}_\mathrm{nod}\) and the definition of the Nehari nodal set \(\mathcal{N}_\mathrm{nod}\).
\end{proof}

\end{document}